\documentclass{amsart}
\usepackage{a4}
\usepackage{hyperref}
\pdfoutput=1

\usepackage{textcmds}

\usepackage[backrefs,lite]{amsrefs}

\usepackage{amssymb}              
\usepackage{lmodern,bm}              
\usepackage{pifont}
\usepackage{mathtools}

\usepackage{mathtools} 
\usepackage{amsfonts}

\usepackage{tikz}		
\usepackage{tikz-3dplot}
		\usetikzlibrary{calc}
		\usetikzlibrary{3d}
\usepackage{graphics}

\usetikzlibrary{patterns,matrix,arrows,shapes}
\usepackage{caption}

\usepackage{tikz-cd}

\usepackage{longtable}
\usepackage{array}
\usepackage{multicol}
\usepackage{multirow}
\usepackage{makecell}
\usepackage{booktabs}
\newcolumntype{C}{>{$}c<{$}}
\newcolumntype{L}{>{$}l<{$}}
\newcolumntype{R}{>{$}r<{$}}

\usepackage{xcolor}

\newtheorem{theorem}{Theorem}[section]

\newtheorem{lemma}[theorem]{Lemma}
\newtheorem{proposition}[theorem]{Proposition}
\newtheorem{corollary}[theorem]{Corollary}

\theoremstyle{definition}

\newtheorem{definition}[theorem]{Definition}
\newtheorem{construction}[theorem]{Construction}

\newtheorem{example}[theorem]{Example}
\newtheorem{remark}[theorem]{Remark}
\newtheorem{setting}[theorem]{Setting}

\numberwithin{equation}{theorem}

\def\CC{{\mathbb C}}
\def\KK{{\mathbb K}}
\def\TT{{\mathbb T}}
\def\ZZ{{\mathbb Z}}

\def\QQ{{\mathbb Q}}
\def\PP{{\mathbb P}}

\def\KKK{\mathcal{K}}

\def\OOO{{\mathcal O}}
\def\RRR{{\mathcal R}}

\def\bangle#1{{\langle #1 \rangle}}

\DeclareMathOperator{\Cl}{\mathrm{Cl}}
\DeclareMathOperator{\cone}{\mathrm{cone}}
\DeclareMathOperator{\conv}{\mathrm{conv}}

\DeclareMathOperator{\im}{\mathrm{im}}
\DeclareMathOperator{\lcm}{\mathrm{lcm}}
\DeclareMathOperator{\lin}{\mathrm{lin}}

\DeclareMathOperator{\Spec}{\mathrm{Spec}}

\DeclareMathOperator{\trop}{\mathrm{trop}}

\newcommand{\linmin}{-3}%
\newcommand{\linmax}{4}%
\newcommand{\leafzero}{2.75}%
\newcommand{\leafone}{3.5}%
\newcommand{\leaftwo}{4.5}%
\tdplotsetmaincoords{70}{45}%

\title[On a combinatorial description of the Gorenstein Index]{On a combinatorial description of the Gorenstein Index for varieties with torus action}

\author[Philipp Iber, Eva Reinert, Milena Wrobel]{Philipp Iber, Eva Reinert, Milena Wrobel}

\address{Institut f\"ur Mathematik, Universit\"at Oldenburg,
26111 Oldenburg, Germany}
\email{philipp.iber@uni-oldenburg.de}
\address{Institut f\"ur Mathematik, Universit\"at Oldenburg,
26111 Oldenburg, Germany}
\email{eva.reinert1@uni-oldenburg.de}
\address{Institut f\"ur Mathematik, Universit\"at Oldenburg,
26111 Oldenburg, Germany}
\email{milena.wrobel@uni-oldenburg.de}


\begin{document}
\bibliographystyle{abbrv}

\begin{abstract}
The anticanonical complex is a combinatorial tool that was invented to extend the features of the Fano polytope from toric geometry to wider classes of varieties.
In this note we show that the Gorenstein index of Fano varieties with torus action of complexity one (and even more general of the so-called general arrangement varieties) can be read off its anticanonical complex in terms of lattice distances in full analogy to the toric Fano polytope. As an application we give concrete bounds on the defining data
of almost homogeneous Fano threefolds of Picard number one having a reductive automorphism group with two-dimensional maximal torus depending on their Gorenstein index.
\end{abstract}

\maketitle

\section{Introduction}
The main objective of this article is to contribute to the development of combinatorial methods for the study of geometric properties of Fano varieties. 
The model case is toric geometry. Here we have the well-known one-to-one correspondence between toric Fano varieties $X$ and the so-called \emph{Fano polytopes} $A_X$. 
These polytopes allow to describe several algebraic and geometric invariants of the corresponding toric varieties in a purely combinatorial manner. One of these invariants is the \emph{Gorenstein index}, that is the smallest positive integer $\iota_X$ such that $\iota_X$-times the canonical divisor $\mathcal{K}_X$ of $X$ is Cartier. In the toric case, this invariant is encoded in the lattice distances of the facets of the Fano polytope, see \cite{cls}. This fact has been used by several authors to contribute to the classification of toric Fano varieties of low Gorenstein index; see \cites{con, kasNil,BatBo,Ath,BrunsRoemer,BaNill}.

The purpose of this note is to generalize this combinatorial Gorenstein criterion 
to Fano varieties $X$ with an effective action of an algebraic torus $\TT$ of higher complexity, where the latter means that the difference $\mathrm{dim}(X) - \mathrm{dim}(\TT)$ is greater or equal to one. More precisely we consider Fano \emph{general arrangement varieties} as introduced in \cite{wrobel:ta_hc}. These are varieties
$X$ coming with a torus action of arbitrary complexity $c$ that gives rise to a specific rational quotient $X \dashrightarrow \PP^c$,
the so called \emph{maximal orbit quotient}, whose critical values form a general hyperplane arrangement.
Note that this class comprises i.a.~all Fano varieties with torus action of complexity one as well as all toric varieties.

By realizing the Fano general arrangement varieties $X$ as subvarieties of toric varieties $Z$, we
can replace the toric Fano polytope with a polyhedral complex, the so-called \emph{anticanonical complex} $\mathcal{A}_X$, which is supported on the tropical variety of $X \subseteq Z$. 

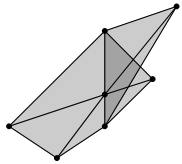
\begin{center}
\begin{tikzpicture}[tdplot_main_coords,
			edge/.style={},
			leaf0/.style={},
			leaf1/.style={},
			leaf2/.style={},
			leaf0e/.style={},
			leaf1e/.style={},
			leaf2e/.style={},
			le/.style={},
			axis/.style={},
			scale=.45]
	\coordinate (o) at (0,0,0);
	\coordinate (leaf0m) at (-\leafzero,-\leafzero,\linmin);
	\coordinate (leaf0p) at (-\leafzero,-\leafzero,\linmax);
	\coordinate (leaf0vm) at (\leafzero,\leafzero,\linmin);
	\coordinate (leaf0vp) at (\leafzero,\leafzero,\linmax);
	\coordinate (linm) at (0,0,\linmin);
	\coordinate (linp) at (0,0,\linmax);
	\coordinate (leaf1m) at (\leafone,0,\linmin);
	\coordinate (leaf1p) at (\leafone,0,\linmax);
	\coordinate (leaf1vm) at (-\leafone,0,\linmin);
	\coordinate (leaf1vp) at (-\leafone,0,\linmax);
	\coordinate (leaf2m) at (0,\leaftwo,\linmin);
	\coordinate (leaf2p) at (0,\leaftwo,\linmax);
	\coordinate (leaf2vm) at (0,-\leaftwo,\linmin);
	\coordinate (leaf2vp) at (0,-\leaftwo,\linmax);
	\coordinate (v01) at (-1,-1,-2);
	\coordinate (v02) at (-2,-2,-1);
	\coordinate (v1) at (2,0,1);
	\coordinate (v2) at (0,3,2);
	\coordinate (e1) at (0,0,2);
	\coordinate (e2) at (0,0,-1);
	
	
	\draw[edge,leaf2e] (o) -- (v2);
	\draw[edge,leaf2e] (v2) -- (e1);
	\draw[edge,leaf2e] (v2) -- (e2);
	\filldraw[leaf2e,opacity=.2] (e1) -- (e2) -- (v2) -- cycle;
	
	
	\filldraw[leaf2e] (v2) circle (2pt);
	
	\draw[edge,leaf1e] (o) -- (v1);
	\draw[edge,leaf1e] (v1) -- (e1);
	\draw[edge,leaf1e] (v1) -- (e2);
	\filldraw[leaf1e,opacity=.2] (e1) -- (e2) -- (v1) -- cycle;
	
	
	\filldraw[] (v1) circle (2pt);  
	
	
	\draw[edge,leaf0e] (o) -- (v01);
	\draw[edge,leaf0e] (o) -- (v02);
	\draw[edge,leaf0e] (v02) -- (e1);
	\draw[edge,leaf0e] (v01) -- (e2);
	\draw[edge,leaf0e] (v01) -- (v02);
	\filldraw[leaf0e,opacity=.2] (e1) -- (o) -- (v02) -- cycle;
	\filldraw[leaf0e,opacity=.2] (v01) -- (o) -- (v02) -- cycle;
	\filldraw[leaf0e,opacity=.2] (e2) -- (o) -- (v01) -- cycle;
	
	
	\filldraw[] (v01) circle (2pt);
	\filldraw[] (v02) circle (2pt);
	
	\draw[edge] (o) -- (e1);
	\draw[edge] (o) -- (e2);
	
	
	\filldraw[] (o) circle (2pt);
	\filldraw[] (e1) circle (2pt);
	\filldraw[] (e2) circle (2pt);
	
	\draw[] (linm) node{\small{$\mathcal{A}_X$ for $X=V(T_0^2T_1 + T_2^2 + T_3^3)\subseteq \PP_{5,8,9,6}$}};
	
\end{tikzpicture}	
\end{center}

The anticanonical complex was
introduced in \cite{bhhn} for varieties with torus action of complexity one and later generalized i.a.\ to the general arrangement case in \cites{hw19,hmw} and has so far been successfully used 
for the classification of singular Fano varieties; see \cites{abhw,bhhn, bh, hmw, h_quadrics, hw19}.
Our main result states that for Fano general arrangement varieties the Gorenstein index can be read off the anticanonical complex 
in full analogy to the toric Fano polytope case:
\begin{theorem}\label{thm:intro}
    Let $X \subseteq Z$ be a Fano general arrangement variety with anticanonical complex $\mathcal{A}_X$. Then the Gorenstein index $\iota_X$ of $X$ equals the least common multiple of the lattice distances of the maximal cells in the boundary of $\mathcal{A}_X:$
    $$
    \iota_X := \mathrm{lcm}(d(0, F); \ F \in \partial\mathcal{A}_X).
    $$
\end{theorem}

As an application of our result
we consider the $\QQ$-factorial rational almost homogeneous Fano varieties of Picard number one with reductive automorphism group having a maximal torus of dimension two that were described in \cite{ahhl14}. Recall that these varieties are uniquely determined up to isomorphy by their divisor class group graded Cox ring and their anticanonical class, see \ref{rem:CoxRingAndAmpleClass}. For fixed Gorenstein index we give concrete bounds on these defining data of the varieties, see Propositions \ref{prop:set1}, \ref{prop:set2}, \ref{prop:set3}, \ref{prop:set4} and~\ref{prop:set5}.

This allows in particular a computer aided search for varieties of small Gorenstein index.
For illustration we list all $\QQ$-factorial rational almost homogeneous Fano varieties of Picard number one with reductive automorphism group having a maximal torus of dimension two and Gorenstein index smaller or equal to three~here:

\begin{corollary}\label{cor:Classification}
Every three-dimensional $\QQ$-factorial Fano variety of Picard number one with reductive automorphism group having a maximal torus of dimension two and Gorenstein index $\iota_X$ smaller or equal to three is isomorphic to one of the following varieties $X$, specified by their $\mathrm{Cl}(X)$ graded Cox ring $\mathcal{R}(X)$, a matrix $[w_1\ldots,w_r]$ of generator degrees and their anticanonical classes $-\mathcal{K}_X$ as follows:

\newcounter{gorindvar}
\setcounter{gorindvar}{0}
\newcommand{\gorno}{\refstepcounter{gorindvar}\thegorindvar}
{\setlength{\tabcolsep}{5pt}\setlength{\arraycolsep}{2pt}
\begin{longtable}{cccccc}
No.
&
$\RRR(X)$
&
$\Cl(X)$
&
$[w_1,\ldots,w_r]$
&
$-\KKK_X$
&
$\iota_X$
\\
\toprule
$\gorno$\label{arzhantsev-et-al-Prop-8-6-iv} & 
$\frac{\CC[T_1,T_2,T_3,T_4,T_5]}{\bangle{T_{1}^{2}+T_{2} T_{3}+T_{4} T_{5}}}$ &
$\ZZ$ &
{\tiny $\left[\begin{array}{ccccc}
1 & 1 & 1 & 1 & 1 
\end{array}\right]$} &
{\tiny $\left[\begin{array}{c}
3 
\end{array}\right]$} &
$1$
\\\midrule
$\gorno$ & 
$\frac{\CC[T_1,T_2,T_3,T_4,T_5]}{\bangle{T_{1} T_{2}+T_{3} T_{4}+T_{5}^{2}}}$ &
$\ZZ \times \ZZ_3$ &
{\tiny $\left[\begin{array}{ccccc}
1 & 1 & 1 & 1 & 1 
\\
 \overline{1} & \overline{2} & \overline{2} & \overline{1} & \overline{0} 
\end{array}\right]$} &
{\tiny $\left[\begin{array}{c}
3 
\\
 \overline{0} 
\end{array}\right]$} &
$1$ 
\\\midrule
$\gorno$ & 
$\frac{\CC[T_1,T_2,T_3,T_4,T_5]}{\bangle{T_{1}^{2}+T_{2} T_{3}+T_{4}^{3}}}$ &
$\ZZ$ &
{\tiny $\left[\begin{array}{ccccc}
3 & 3 & 3 & 2 & 1 
\end{array}\right]$} &
{\tiny $\left[\begin{array}{c}
6 
\end{array}\right]$} &
$1$
\\\midrule
$\gorno$ & 
$\frac{\CC[T_1,T_2,T_3,T_4,T_5]}{\bangle{T_{1}^{2}+T_{2} T_{3}+T_{4}^{4}}}$ &
$\ZZ \times \ZZ_2$ &
{\tiny $\left[\begin{array}{ccccc}
2 & 2 & 2 & 1 & 1 
\\
 \overline{1} & \overline{1} & \overline{1} & \overline{1} & \overline{0} 
\end{array}\right]$} &
{\tiny $\left[\begin{array}{c}
4 
\\
 \overline{0} 
\end{array}\right]$} &
$1$
\\\midrule
\midrule
$\gorno$ & 
$\frac{\CC[T_1,T_2,T_3,T_4,T_5]}{\bangle{T_{1} T_{2}+T_{3} T_{4}+T_{5}^{3}}}$ &
$\ZZ \times \ZZ_4$ &
{\tiny $\left[\begin{array}{ccccc}
1 & 2 & 2 & 1 & 1 
\\
 \overline{1} & \overline{3} & \overline{3} & \overline{1} & \overline{0} 
\end{array}\right]$} &
{\tiny $\left[\begin{array}{c}
4 
\\
 \overline{0} 
\end{array}\right]$} &
$2$
\\\midrule
$\gorno$ & 
$\frac{\CC[T_1,T_2,T_3,T_4,T_5]}{\bangle{T_{1}^{2}+T_{2} T_{3}+T_{4}^{2} T_{5}^{2}}}$ &
$\ZZ \times \ZZ_2$ &
{\tiny $\left[\begin{array}{ccccc}
2 & 2 & 2 & 1 & 1 
\\
 \overline{1} & \overline{1} & \overline{1} & \overline{0} & \overline{0} 
\end{array}\right]$} &
{\tiny $\left[\begin{array}{c}
4 
\\
 \overline{1} 
\end{array}\right]$} &
$2$
\\\midrule
$\gorno$ & 
$\frac{\CC[T_1,T_2,T_3,T_4,T_5]}{\bangle{T_{1}^{2}+T_{2} T_{3}+T_{4}^{2} T_{5}^{2}}}$ &
$\ZZ \times \ZZ_4$ &
{\tiny $\left[\begin{array}{ccccc}
2 & 2 & 2 & 1 & 1 
\\
 \overline{3} & \overline{3} & \overline{3} & \overline{1} & \overline{0} 
\end{array}\right]$} &
{\tiny $\left[\begin{array}{c}
4 
\\
 \overline{0} 
\end{array}\right]$} &
$2$
\\\midrule
$\gorno$ & 
$\frac{\CC[T_1,T_2,T_3,T_4,T_5]}{\bangle{T_{1}^{2}+T_{2} T_{3}+T_{4}^{3} T_{5}^{5}}}$ &
$\ZZ$ &
{\tiny $\left[\begin{array}{ccccc}
4 & 4 & 4 & 1 & 1 
\end{array}\right]$} &
{\tiny $\left[\begin{array}{c}
6 
\end{array}\right]$} &
$2$
\\\midrule

$\gorno$ & 
$\frac{\CC[T_1,T_2,T_3,T_4,T_5]}{\bangle{T_{1}^{2}+T_{2} T_{3}+T_{4}^{3} T_{5}^{7}}}$ &
$\ZZ$ &
{\tiny $\left[\begin{array}{ccccc}
8 & 8 & 8 & 3 & 1 
\end{array}\right]$} &
{\tiny $\left[\begin{array}{c}
12 
\end{array}\right]$} &
$2$
\\\midrule
$\gorno$ & 
$\frac{\CC[T_1,T_2,T_3,T_4,T_5]}{\bangle{T_{1}^{2}+T_{2} T_{3}+T_{4} T_{5}^{3}}}$ &
$\ZZ$ &
{\tiny $\left[\begin{array}{ccccc}
4 & 4 & 4 & 5 & 1 
\end{array}\right]$} &
{\tiny $\left[\begin{array}{c}
10 
\end{array}\right]$} &
$2$
\\\midrule
$\gorno$ & 
$\frac{\CC[T_1,T_2,T_3,T_4,T_5]}{\bangle{T_{1}^{2}+T_{2} T_{3}+T_{4} T_{5}^{3}}}$ &
$\ZZ \times \ZZ_2$ &
{\tiny $\left[\begin{array}{ccccc}
2 & 2 & 2 & 1 & 1 
\\
 \overline{1} & \overline{1} & \overline{1} & \overline{0} & \overline{0} 
\end{array}\right]$} &
{\tiny $\left[\begin{array}{c}
4 
\\
 \overline{1} 
\end{array}\right]$} &
$2$
\\\midrule
$\gorno$ & 
$\frac{\CC[T_1,T_2,T_3,T_4,T_5]}{\bangle{T_{1}^{2}+T_{2} T_{3}+T_{4} T_{5}^{5}}}$ &
$\ZZ$ &
{\tiny $\left[\begin{array}{ccccc}
8 & 8 & 8 & 1 & 3 
\end{array}\right]$} &
{\tiny $\left[\begin{array}{c}
12 
\end{array}\right]$} &
$2$
\\\midrule
$\gorno$ & 
$\frac{\CC[T_1,T_2,T_3,T_4,T_5]}{\bangle{T_{1}^{2}+T_{2} T_{3}+T_{4} T_{5}^{7}}}$ &
$\ZZ \times \ZZ_3$ &
{\tiny $\left[\begin{array}{ccccc}
4 & 4 & 4 & 1 & 1 
\\
 \overline{2} & \overline{2} & \overline{2} & \overline{1} & \overline{0} 
\end{array}\right]$} &
{\tiny $\left[\begin{array}{c}
6 
\\
 \overline{0} 
\end{array}\right]$} &
$2$
\\\midrule
$\gorno$ & 
$\frac{\CC[T_1,T_2,T_3,T_4,T_5]}{\bangle{T_{1}^{2}+T_{2} T_{3}+T_{4}^{6}}}$ &
$\ZZ \times \ZZ_2$ &
{\tiny $\left[\begin{array}{ccccc}
3 & 3 & 3 & 1 & 2 
\\
 \overline{1} & \overline{1} & \overline{1} & \overline{0} & \overline{1} 
\end{array}\right]$} &
{\tiny $\left[\begin{array}{c}
6 
\\
 \overline{0} 
\end{array}\right]$} &
$2$
\\\midrule
\midrule
$\gorno$ & 
$\frac{\CC[T_1,T_2,T_3,T_4,T_5]}{\bangle{T_{1} T_{2}+T_{3} T_{4}+T_{5}^{5}}}$ &
$\ZZ \times \ZZ_3$ &
{\tiny $\left[\begin{array}{ccccc}
2 & 3 & 3 & 2 & 1 
\\
 \overline{1} & \overline{2} & \overline{2} & \overline{1} & \overline{0} 
\end{array}\right]$} &
{\tiny $\left[\begin{array}{c}
6 
\\
 \overline{0} 
\end{array}\right]$} &
$3$
\\\midrule
$\gorno$ & 
$\frac{\CC[T_1,T_2,T_3,T_4,T_5]}{\bangle{T_{1} T_{2}+T_{3} T_{4}+T_{5}^{4}}}$ &
$\ZZ \times \ZZ_5$ &
{\tiny $\left[\begin{array}{ccccc}
1 & 3 & 3 & 1 & 1 
\\
 \overline{1} & \overline{4} & \overline{4} & \overline{1} & \overline{0} 
\end{array}\right]$} &
{\tiny $\left[\begin{array}{c}
5 
\\
 \overline{0} 
\end{array}\right]$} &
$3$
\\\midrule
$\gorno$ & 
$\frac{\CC[T_1,T_2,T_3,T_4,T_5]}{\bangle{T_{1} T_{2}+T_{3} T_{4}+T_{5}^{2}}}$ &
$\ZZ \times \ZZ_3$ &
{\tiny $\left[\begin{array}{ccccc}
1 & 3 & 3 & 1 & 2 
\\
 \overline{0} & \overline{1} & \overline{1} & \overline{0} & \overline{2} 
\end{array}\right]$} &
{\tiny $\left[\begin{array}{c}
6 
\\
 \overline{0} 
\end{array}\right]$} &
$3$
\\\midrule
$\gorno$ & 
$\frac{\CC[T_1,T_2,T_3,T_4,T_5]}{\bangle{T_{1} T_{2}+T_{3} T_{4}+T_{5}^{2}}}$ &
$\ZZ \times \ZZ_3$ &
{\tiny $\left[\begin{array}{ccccc}
3 & 1 & 1 & 3 & 2 
\\
 \overline{2} & \overline{0} & \overline{0} & \overline{2} & \overline{1} 
\end{array}\right]$} &
{\tiny $\left[\begin{array}{c}
6 
\\
 \overline{0} 
\end{array}\right]$} &
$3$
\\\midrule
$\gorno$ & 
$\frac{\CC[T_1,T_2,T_3,T_4,T_5]}{\bangle{T_{1} T_{2}+T_{3} T_{4}+T_{5}^{2}}}$ &
$\ZZ \times \ZZ_9$ &
{\tiny $\left[\begin{array}{ccccc}
1 & 1 & 1 & 1 & 1 
\\
 \overline{1} & \overline{8} & \overline{8} & \overline{1} & \overline{0} 
\end{array}\right]$} &
{\tiny $\left[\begin{array}{c}
3 
\\
 \overline{0} 
\end{array}\right]$} &
$3$
\\\midrule
$\gorno$ & 
$\frac{\CC[T_1,T_2,T_3,T_4,T_5]}{\bangle{T_{1}^{2}+T_{2} T_{3}+T_{4}^{2} T_{5}^{7}}}$ &
$\ZZ$ &
{\tiny $\left[\begin{array}{ccccc}
15 & 15 & 15 & 1 & 4 
\end{array}\right]$} &
{\tiny $\left[\begin{array}{c}
20 
\end{array}\right]$} &
$3$
\\\midrule
$\gorno$ & 
$\frac{\CC[T_1,T_2,T_3,T_4,T_5]}{\bangle{T_{1}^{2}+T_{2} T_{3}+T_{4}^{2} T_{5}^{8}}}$ &
$\ZZ \times \ZZ_2$ &
{\tiny $\left[\begin{array}{ccccc}
9 & 9 & 9 & 1 & 2 
\\
 \overline{1} & \overline{1} & \overline{1} & \overline{0} & \overline{1} 
\end{array}\right]$} &
{\tiny $\left[\begin{array}{c}
12 
\\
 \overline{0} 
\end{array}\right]$} &
$3$
\\\midrule
$\gorno$ & 
$\frac{\CC[T_1,T_2,T_3,T_4,T_5]}{\bangle{T_{1}^{2}+T_{2} T_{3}+T_{4}^{2} T_{5}^{10}}}$ &
$\ZZ \times \ZZ_4$ &
{\tiny $\left[\begin{array}{ccccc}
6 & 6 & 6 & 1 & 1 
\\
 \overline{3} & \overline{3} & \overline{3} & \overline{1} & \overline{0} 
\end{array}\right]$} &
{\tiny $\left[\begin{array}{c}
8 
\\
 \overline{0} 
\end{array}\right]$} &
$3$
\\\midrule
$\gorno$ & 
$\frac{\CC[T_1,T_2,T_3,T_4,T_5]}{\bangle{T_{1}^{2}+T_{2} T_{3}+T_{4}^{3} T_{5}^{9}}}$ &
$\ZZ$ &
{\tiny $\left[\begin{array}{ccccc}
6 & 6 & 6 & 1 & 1 
\end{array}\right]$} &
{\tiny $\left[\begin{array}{c}
8 
\end{array}\right]$} &
$3$
\\\midrule
$\gorno$ & 
$\frac{\CC[T_1,T_2,T_3,T_4,T_5]}{\bangle{T_{1}^{2}+T_{2} T_{3}+T_{4}^{3} T_{5}^{12}}}$ &
$\ZZ$ &
{\tiny $\left[\begin{array}{ccccc}
9 & 9 & 9 & 2 & 1 
\end{array}\right]$} &
{\tiny $\left[\begin{array}{c}
12 
\end{array}\right]$} &
$3$
\\\midrule
$\gorno$ & 
$\frac{\CC[T_1,T_2,T_3,T_4,T_5]}{\bangle{T_{1}^{2}+T_{2} T_{3}+T_{4}^{3} T_{5}^{18}}}$ &
$\ZZ$ &
{\tiny $\left[\begin{array}{ccccc}
15 & 15 & 15 & 4 & 1 
\end{array}\right]$} &
{\tiny $\left[\begin{array}{c}
20 
\end{array}\right]$} &
$3$
\\\midrule
$\gorno$ & 
$\frac{\CC[T_1,T_2,T_3,T_4,T_5]}{\bangle{T_{1}^{2}+T_{2} T_{3}+T_{4}^{3} T_{5}^{3}}}$ &
$\ZZ$ &
{\tiny $\left[\begin{array}{ccccc}
3 & 3 & 3 & 1 & 1 
\end{array}\right]$} &
{\tiny $\left[\begin{array}{c}
5 
\end{array}\right]$} &
$3$
\\\midrule
$\gorno$ & 
$\frac{\CC[T_1,T_2,T_3,T_4,T_5]}{\bangle{T_{1}^{2}+T_{2} T_{3}+T_{4}^{3} T_{5}^{3}}}$ &
$\ZZ$ &
{\tiny $\left[\begin{array}{ccccc}
9 & 9 & 9 & 1 & 5 
\end{array}\right]$} &
{\tiny $\left[\begin{array}{c}
15 
\end{array}\right]$} &
$3$
\\\midrule
$\gorno$ & 
$\frac{\CC[T_1,T_2,T_3,T_4,T_5]}{\bangle{T_{1}^{2}+T_{2} T_{3}+T_{4}^{3} T_{5}^{3}}}$ &
$\ZZ \times \ZZ_5$ &
{\tiny $\left[\begin{array}{ccccc}
3 & 3 & 3 & 1 & 1 
\\
 \overline{2} & \overline{2} & \overline{2} & \overline{3} & \overline{0} 
\end{array}\right]$} &
{\tiny $\left[\begin{array}{c}
5 
\\
 \overline{0} 
\end{array}\right]$} &
$3$
\\\midrule
$\gorno$ & 
$\frac{\CC[T_1,T_2,T_3,T_4,T_5]}{\bangle{T_{1}^{2}+T_{2} T_{3}+T_{4}^{4} T_{5}^{7}}}$ &
$\ZZ$ &
{\tiny $\left[\begin{array}{ccccc}
9 & 9 & 9 & 1 & 2 
\end{array}\right]$} &
{\tiny $\left[\begin{array}{c}
12 
\end{array}\right]$} &
$3$
\\\midrule
$\gorno$ & 
$\frac{\CC[T_1,T_2,T_3,T_4,T_5]}{\bangle{T_{1}^{2}+T_{2} T_{3}+T_{4}^{4} T_{5}^{8}}}$ &
$\ZZ \times \ZZ_2$ &
{\tiny $\left[\begin{array}{ccccc}
6 & 6 & 6 & 1 & 1 
\\
 \overline{1} & \overline{1} & \overline{1} & \overline{1} & \overline{0} 
\end{array}\right]$} &
{\tiny $\left[\begin{array}{c}
8 
\\
 \overline{0} 
\end{array}\right]$} &
$3$
\\\midrule
$\gorno$ & 
$\frac{\CC[T_1,T_2,T_3,T_4,T_5]}{\bangle{T_{1}^{2}+T_{2} T_{3}+T_{4}^{4} T_{5}^{10}}}$ &
$\ZZ \times \ZZ_2$ &
{\tiny $\left[\begin{array}{ccccc}
9 & 9 & 9 & 2 & 1 
\\
 \overline{1} & \overline{1} & \overline{1} & \overline{1} & \overline{0} 
\end{array}\right]$} &
{\tiny $\left[\begin{array}{c}
12 
\\
 \overline{0} 
\end{array}\right]$} &
$3$
\\\midrule
$\gorno$ & 
$\frac{\CC[T_1,T_2,T_3,T_4,T_5]}{\bangle{T_{1}^{2}+T_{2} T_{3}+T_{4}^{5} T_{5}^{7}}}$ &
$\ZZ$ &
{\tiny $\left[\begin{array}{ccccc}
6 & 6 & 6 & 1 & 1 
\end{array}\right]$} &
{\tiny $\left[\begin{array}{c}
8 
\end{array}\right]$} &
$3$
\\\midrule
$\gorno$ & 
$\frac{\CC[T_1,T_2,T_3,T_4,T_5]}{\bangle{T_{1}^{2}+T_{2} T_{3}+T_{4}^{5} T_{5}^{8}}}$ &
$\ZZ$ &
{\tiny $\left[\begin{array}{ccccc}
9 & 9 & 9 & 2 & 1 
\end{array}\right]$} &
{\tiny $\left[\begin{array}{c}
12 
\end{array}\right]$} &
$3$
\\\midrule
$\gorno$ & 
$\frac{\CC[T_1,T_2,T_3,T_4,T_5]}{\bangle{T_{1}^{2}+T_{2} T_{3}+T_{4}^{5} T_{5}^{10}}}$ &
$\ZZ$ &
{\tiny $\left[\begin{array}{ccccc}
15 & 15 & 15 & 4 & 1 
\end{array}\right]$} &
{\tiny $\left[\begin{array}{c}
20 
\end{array}\right]$} &
$3$
\\\midrule
$\gorno$ & 
$\frac{\CC[T_1,T_2,T_3,T_4,T_5]}{\bangle{T_{1}^{2}+T_{2} T_{3}+T_{4} T_{5}^{4}}}$ &
$\ZZ \times \ZZ_3$ &
{\tiny $\left[\begin{array}{ccccc}
3 & 3 & 3 & 2 & 1 
\\
 \overline{2} & \overline{2} & \overline{2} & \overline{1} & \overline{0} 
\end{array}\right]$} &
{\tiny $\left[\begin{array}{c}
6 
\\
 \overline{0} 
\end{array}\right]$} &
$3$
\\\midrule
$\gorno$ & 
$\frac{\CC[T_1,T_2,T_3,T_4,T_5]}{\bangle{T_{1}^{2}+T_{2} T_{3}+T_{4} T_{5}^{2}}}$ &
$\ZZ$ &
{\tiny $\left[\begin{array}{ccccc}
3 & 3 & 3 & 4 & 1 
\end{array}\right]$} &
{\tiny $\left[\begin{array}{c}
8 
\end{array}\right]$} &
$3$
\\\midrule
$\gorno$ & 
$\frac{\CC[T_1,T_2,T_3,T_4,T_5]}{\bangle{T_{1}^{2}+T_{2} T_{3}+T_{4} T_{5}^{5}}}$ &
$\ZZ$ &
{\tiny $\left[\begin{array}{ccccc}
6 & 6 & 6 & 7 & 1 
\end{array}\right]$} &
{\tiny $\left[\begin{array}{c}
14 
\end{array}\right]$} &
$3$
\\\midrule
$\gorno$ & 
$\frac{\CC[T_1,T_2,T_3,T_4,T_5]}{\bangle{T_{1}^{2}+T_{2} T_{3}+T_{4}^{5}}}$ &
$\ZZ$ &
{\tiny $\left[\begin{array}{ccccc}
5 & 5 & 5 & 2 & 3 
\end{array}\right]$} &
{\tiny $\left[\begin{array}{c}
10 
\end{array}\right]$} &
$3$
\\\midrule
$\gorno$ & 
$\frac{\CC[T_1,T_2,T_3,T_4,T_5]}{\bangle{T_{1}^{2}+T_{2} T_{3}+T_{4}^{8}}}$ &
$\ZZ \times \ZZ_2$ &
{\tiny $\left[\begin{array}{ccccc}
4 & 4 & 4 & 1 & 3 
\\
 \overline{1} & \overline{1} & \overline{1} & \overline{0} & \overline{1} 
\end{array}\right]$} &
{\tiny $\left[\begin{array}{c}
8 
\\
 \overline{0} 
\end{array}\right]$} &
$3$
\\\midrule
$\gorno$ & 
$\frac{\CC[T_1,T_2,T_3,T_4,T_5]}{\bangle{T_{1}^{2}+T_{2} T_{3}+T_{4}^{9}}}$ &
$\ZZ$ &
{\tiny $\left[\begin{array}{ccccc}
9 & 9 & 9 & 2 & 1 
\end{array}\right]$} &
{\tiny $\left[\begin{array}{c}
12 
\end{array}\right]$} &
$3$
\\\midrule
$\gorno$ & 
$\frac{\CC[T_1,T_2,T_3,T_4,T_5]}{\bangle{T_{1}^{2}+T_{2} T_{3}+T_{4}^{12}}}$ &
$\ZZ \times \ZZ_2$ &
{\tiny $\left[\begin{array}{ccccc}
6 & 6 & 6 & 1 & 1 
\\
 \overline{1} & \overline{1} & \overline{1} & \overline{1} & \overline{0} 
\end{array}\right]$} &
{\tiny $\left[\begin{array}{c}
8 
\\
 \overline{0} 
\end{array}\right]$} &
$3$
\\\midrule
$\gorno$ & 
$\frac{\CC[T_1,T_2,T_3,T_4,T_5]}{\bangle{T_{1}^{2}+T_{2} T_{3}+T_{4}^{18}}}$ &
$\ZZ \times \ZZ_2$ &
{\tiny $\left[\begin{array}{ccccc}
9 & 9 & 9 & 1 & 2 
\\
 \overline{1} & \overline{1} & \overline{1} & \overline{0} & \overline{1} 
\end{array}\right]$} &
{\tiny $\left[\begin{array}{c}
12 
\\
 \overline{0} 
\end{array}\right]$} &
$3$
\\
\bottomrule
\end{longtable}
}
\end{corollary}

\tableofcontents

\section{Background on General Arrangement Varieties}
We assume the reader to be familiar with the foundations of toric geometry; see \cites{cls, ful} for introductory texts. In this section we recall the necessary facts and notions on \emph{general arrangement varieties}. 
This class of varieties has been introduced in \cite{wrobel:ta_hc} and can be obtained by using the constructive description of varieties with torus action provided there: The construction is based on a result of \cite{hs} relating the \emph{Cox ring}
$$
\mathcal{R}(X) := \bigoplus_{[D] \in \mathrm{Cl}(X)} \Gamma(X, \mathcal{O}_X(D))
$$
of a variety $X$ with torus action to that of a suitable rational quotient $X \dashrightarrow Y$, the so-called \emph{maximal orbit quotient}. It provides us with the Cox ring and an associated embedding of $X$ into a toric variety $Z_X$. 
Specializing the procedure to the case that $Y$ is the projective or the affine line, one retrieves the Cox ring based approach to rational varieties with torus action of complexity one developed in \cites{hh, hhs, hw_cpl1}. The class of general arrangement varieties introduced in \cite{wrobel:ta_hc} can then be seen as a controlled step leaving the case of complexity one: These are varieties $X$ with torus action and $Y = \PP^n$ such that the critical values of the maximal orbit quotient $X \dashrightarrow Y$ form a general hyperplane arrangement.

We briefly recall the construction of graded rings $R(A,P)$ that are defined by a pair of matrices and which turn out to be the Cox rings of general arrangement varieties; compare \cite{wrobel:ta_hc}:
\begin{construction}
\label{constr:R(A,P_0)}
Fix integers $r \ge c > 0$
and $n_0, \ldots, n_r > 0$ 
as well as $m \ge 0$. 
Set $n := n_0 + \ldots + n_r$.
For every $i = 0, \ldots, r$ fix a tuple $l_i \in \ZZ_{>0}^{n_i}$ and define a monomial 
$$
T_i^{l_i}
\ := \ 
T_{i1}^{l_{i1}} \cdots T_{in_i}^{l_{in_i}}
\ \in \ 
\KK[T_{ij},S_k; \ 0 \leq i \leq r, \ 1 \leq j \leq n_i, \ 1 \leq k \leq m].
$$
We will also write $\KK[T_{ij}, S_k]$ for the above polynomial ring.
Let $A :=(a_0, \ldots a_r)$ be a a $(c+1) \times (r+1)$ matrix over $\KK$ 
such that any $c+1$ of its columns 
$a_0, \ldots, a_r$ are linearly independent.
For every $t = 1, \ldots, r-c$, 
we obtain a polynomial
$$ 
g_t 
\ := \
\det
\left[
\begin{array}{cccc}
a_0 & \ldots & a_c & a_{c+t}
\\
T_0^{l_0} & \ldots & T_c^{l_{c}} & T_{c+t}^{l_{c+t}}
\end{array}
\right]
\ \in \ 
\KK[T_{ij},S_k].
$$
In the next step, we construct a grading on the factor ring
$$
\KK[T_{ij}, S_k]/\bangle{g_1, \ldots, g_{r-c}}.
$$
We build up 
an integral $(r+s) \times (n+m)$-matrix $P$ from an
$r \times (n+m)$-matrix matrix $P_0$
built from the tuples of positive integers 
$l_i$, where $i = 0, \ldots, r$ and a $s \times (n+m)$-matrix $D$
as follows
$$
P:=
\left[
\begin{array}{c}
P_0\\
\hline
D
\end{array}
\right]
\ := \
\left[
\begin{array}{ccccccc}
-l_{0} & l_{1} &  & 0 & 0  &  \ldots & 0
\\
\vdots & \vdots & \ddots & \vdots & \vdots &  & \vdots
\\
-l_{0} & 0 &  & l_{r} & 0  &  \ldots & 0\\
\hline
&&&&&&
\\
&&&D&&&
\end{array}
\right],
$$
whereby we require the columns of the matrix $P$ to be pairwise different, primitive and generate $\QQ^{r+s}$ as a vector space.

Now, let $e_{ij} \in \ZZ^{n}$ 
and $e_k \in \ZZ^{m}$ denote the 
canonical basis vectors
and consider the projection 
$$
Q \colon \ZZ^{n+m} 
\ \to \ 
K := \ZZ^{n+m} / \im(P^*)
$$ 
onto the factor group
by the row lattice of $P$.
Then the 
\emph{$K$-graded $\KK$-algebra
associated with $(A,P)$} 
is defined as
$$ 
R(A,P)
\ := \ 
\KK[T_{ij},S_k] / \bangle{g_1,\ldots,g_{r-c}},
$$
$$
\deg(T_{ij}) :=  Q(e_{ij}),
\qquad
\deg(S_k) :=  Q(e_k).
$$
We note that the rings $R(A,P)$ can be directly read off the matrices $A$ and $P$. They are integral normal complete intersections.
\end{construction}

From a ring $R(A,P)$ as above, we obtain general arrangement varieties $X$ together with an embedding $X \subseteq Z$ into a toric variety $Z$ via the following construction:

\begin{construction}\label{constr:XAPSigma}
Let $R(A,P)$ be as above. The generators $T_{ij},S_k$ of $R(A,P)$ give rise to an embedding
\begin{center}
\begin{tikzcd}
\bar X := \Spec(R(A,P))\arrow[r, hook]
&
\bar Z:=\KK^{n+m}. 
\end{tikzcd}
\end{center}
Fix any fan $\Sigma$ in $\QQ^{r+s}$ having the columns of $P$ as its primitive ray generators and denote by $Z$ the toric variety with defining fan $\Sigma$.
Consider the linear map $P\colon\QQ^{n+m}\rightarrow\QQ^{r+s}$ defined by $P$, set
$\hat\Sigma:=\{\sigma\preceq\gamma;\  P(\sigma)\in\Sigma\}$, where $\gamma \subseteq \QQ^{n+m}$ denotes the positive orthant, and denote by $\hat Z$ the corresponding toric variety. Then we obtain a commutative diagram
\begin{center}
\begin{tikzcd}
\bar X\cap \hat Z \arrow[r, hook]\arrow[d,"p"]
&\hat Z\arrow[d,"p"]
\\
X(A,P,\Sigma)\arrow[r, hook]
&
Z
\end{tikzcd}
\end{center}
where $p$ denotes the toric morphism corresponding to the linear map $P$
and $X:=X(A,P,\Sigma)$
is the closure of $p(\bar X\cap \TT^{n+m})$
inside $Z$.
By construction, the variety $X$ is invariant under the subtorus action $\TT^s\subseteq\TT^{r+s}$ of the acting torus of $Z$.
\end{construction}

The varieties $X:=X(A,P,\Sigma)\subseteq Z$
are normal varieties with
dimension, invertible functions, divisor class group and Cox ring given in terms of their defining data by:
$$
\dim(X) = s + c,
\qquad
\Gamma(X,\mathcal{O}^*) = \KK^*,
\qquad
\Cl(X) = K,
\qquad
\mathcal{R}(X) = R(A,P). 
$$
The torus action of $\TT^s$ on $X$ is effective and of complexity $c$, i.e.\ the general torus orbit is of codimension $c$.

\begin{definition}
Let $X:=X(A,P,\Sigma) \subseteq Z$ arise from Construction \ref{constr:XAPSigma}. Then we call $X \subseteq Z$ an \emph{explicit general arrangement variety}. Moreover we call any $\TT$-variety that is equivariantly isomorphic to an explicit general arrangement variety a \emph{general arrangement variety}.
\end{definition}

\begin{remark}\label{rem:CoxRingAndAmpleClass}
Let $R(A,P)$ be as above and assume that the columns of $P$ generate $\QQ^{r+s}$ as a cone. Let $\gamma$ denote the positive orthant $\QQ_{\geq0}^{n+m}$. We define a polyhedral cone 
$$
\mathrm{Mov}(R(A,P)) :=  \bigcap_{\gamma_0\preceq \gamma \text{ facet}} Q(\gamma_0) \subseteq K_\QQ.
$$
Then any projective explicit general arrangement variety is of the form $X(A, P, \Sigma(u))$, where $\Sigma(u)$ is constructed as follows: Let $u \in \mathrm{Mov}(R(A,P))^\circ$ and set
$$
\Sigma(u) := \left\{P(\gamma_0^*); \ \gamma_0 \preceq \gamma, \ u \in Q(\gamma_0)^\circ\right\}, \quad \text{where } \gamma_0^* := \mathrm{cone}(e_i; \ e_i \notin \gamma_0).
$$
In particular, up to isomorphy, a projective general arrangement variety $X$ can be regained from its $\mathrm{Cl}(X)$-graded Cox ring $\mathcal{R}(X)$ and an ample class $u \in \mathrm{Cl}(X)$
in the above way. Moreover, if $X$ is Fano, we may choose $u = - \mathcal{K}_X.$
\end{remark}

\begin{remark}
    Let $X:=X(A,P,\Sigma) \subseteq Z$ be an explicit general arrangement variety. We note that in Construction \ref{constr:XAPSigma} we may successively remove all maximal cones $\sigma \in \Sigma$ whose corresponding orbit does not intersect $X$, that is,
    $$X \cap \TT^{r+s} \cdot z_\sigma = \emptyset,$$
    where $z_\sigma$ denotes the common limit point for $t \rightarrow 0$ of all one-parameter subgroups $t \mapsto (t^{v_1}, \ldots, t^{v_{r+s}})$ of the acting torus $\TT^{r+s}$ on $Z$ with $v \in \ZZ^{r+s}$ taken from the relative interior $\sigma^\circ \subseteq \sigma$.
    We end up with a minimal fan $\Sigma$ still defining the same general arrangement variety $X$.
    We call the toric variety corresponding to this minimal fan the \emph{minimal ambient toric variety} of $X$ and denote it with $Z_X$.
\end{remark}

We end this chapter by a close investigation of the structure of the fan $\Sigma$ of the minimal ambient toric variety $Z_X$ of a general arrangement variety $X$. 

Let us briefly recall the basic notions on tropical varieties.
Let $Z$ be a toric variety with acting torus $\TT$.
For a closed subvariety $X\subseteq Z$ intersecting the torus $\TT$ non trivially consider the vanishing ideal $I(X \cap \TT)$ in the Laurent polynomial ring $\OOO(T)$. For every $f \in I(X\cap \TT)$ let $|\Sigma(f)|$ denote the support of the codimension one skeleton of the normal quasifan of its Newton polytope.
Then the \emph{tropical variety $\trop(X)$ of} $X$ is defined as follows, 
see \cite[Def. 3.2.1]{MaclaganSturmfels}:
$$
\trop(X) := \bigcap_{f \in I(X \cap \TT)} |\Sigma(f)| \subseteq \QQ^{\dim(Z)}.
$$
The following result of Tevelev then gives a criterion on which orbits of the toric variety $Z$ are intersected by the embedded variety $X$ in terms of the tropical variety $\mathrm{trop}(X)$, see \cite{tev}:
\begin{remark}
Let $X\subseteq Z$ be a closed embedding.
Then $X$ intersects the torus orbit $T \cdot z_\sigma$ corresponding to the cone ${\sigma \in \Sigma}$ non-trivially if and only if the relative interior $\sigma^\circ$ intersects the tropical variety $\trop(X)$ non-trivially.
\end{remark}

Using this criterion, we obtain that the cones occurring in the fan corresponding to the minimal ambient toric variety of an explicit general arrangement variety $X \subseteq Z_X$
are as follows:

\begin{remark}\label{rem:tropAndMinimalAmbient}
Let $X(A, P, \Sigma)\subseteq Z_X$ be an explicit general arrangement variety of complexity $c$ and denote with $\Sigma_{\PP^r}$ the fan corresponding to the toric variety $\PP^r$. Then we have
$$
|\mathrm{trop}(X)| = |\Sigma^{\leq c}_{\PP_r}|\times\QQ^s,
\quad
\text{where}
\quad 
\Sigma_{\PP^r}^{\leq c}:=\left\{\sigma\in\Sigma_{\PP^r};\  \dim(\sigma)\leq c\right\}$$
We endow $\mathrm{trop}(X)$ with the following quasifan structure:
Denote by $e_1,\ldots,e_{r+s}$ the canonical basis of $\QQ^{r+s}$ and set $e_0:=-\sum_{i=1}^r e_i$.
For any subset $I\subseteq\{0,\ldots,r\}$  with $0 \leq |I| \leq c$ we set
$$\lambda_I:=\cone(e_i;\ i\in I) + \lin(e_{r+1},\ldots,
e_{r+s}).
$$ 
Then we have $\lambda_I\subseteq\trop(X)$ and these cones define quasifan structure on $\mathrm{trop}(X)$. More precisely we have
$$
\mathrm{trop}(X) = \Sigma^{\leq c}_{\PP_r}\times\QQ^s = \left\{\sigma \times \QQ^s; \ \sigma \in \Sigma^{\leq c}_{\PP_r}\right\} = \left\{\lambda_I; \ I \subseteq \left\{0, \ldots, r\right\}, \ 0 \leq |I| \leq c\right\}.
$$
The cones $\lambda_I$, where $1\leq |I| =: k$ are called the {\em $k$-leaves of $\trop(X)$}.
Moreover, we have the {\em lineality space of $\trop(X)$}: 
$$\lambda_\mathrm{lin}:= \lambda_\emptyset = \bigcap
 \lambda_I = \mathrm{lin}(e_{r+1}, \ldots, e_{r+s}).$$
Using this quasifan structure on $\mathrm{trop}(X)$, we can distinguish between two types of cones that occur in the defining fan $\Sigma$ of the minimal ambient toric variety $Z_X$ of $X$:
A cone $\sigma \in \Sigma$ is either a \emph{leaf cone}, that means, $\sigma \subseteq \lambda_I$ holds for a leaf $\lambda_i \in \mathrm{trop}(X)$, or $\sigma \in \Sigma$ is a \emph{big cone}, that means
$\sigma \cap \lambda_i^\circ \neq \emptyset$ 
holds for all $1$-leaves $\lambda_i$ of $\trop(X)$.
Moreover, we call a big cone \emph{elementary big}, if for every $0 \leq i \leq r$ there exists precisely one ray $\varrho_i$ of $\sigma$ with $\varrho_i \subseteq \lambda_i$.

\end{remark}

\begin{lemma}\label{lem:interiorIntersNotEmpty}
    Let $X:=X(A,P, \Sigma) \subseteq Z_X$ be an explicit general arrangement variety.
    Let $\sigma \in \Sigma$ be a cone with $\sigma \not \subseteq \lambda_\mathrm{lin}$. Then the following statements are equivalent:
    \begin{enumerate}
        \item 
        $\sigma$ is a big cone.
        \item 
        We have $\sigma^\circ \cap \lambda_\mathrm{lin} \neq \emptyset.$
    \end{enumerate}
\end{lemma}
\begin{proof}
    As each cone in $\Sigma_X$ is either a big cone or a leaf cone, we only need to show the implication (i) $\Rightarrow$ (ii).
    So let $\sigma$ be a big cone.
    For $0 \leq i \leq r$ we set
    $$
    J_i := \left\{j \in \left\{1, \ldots, n_i\right\}; \ v_{ij} \in \sigma \right\}
    \quad
    \text{and}
    \quad
    J:= \left\{k \in \left\{1, \ldots, m\right\}; \ v_k \in \sigma\right\}.
    $$
    As $\sigma$ is a big cone, none of the sets $J_i$ is empty.
    We obtain $\alpha_i := \sum_{j \in J_i} l_{ij} > 0$ for all $0 \leq i \leq r$ and conclude
    $$
    \sum_{i=0}^r \frac{\alpha_0}{\alpha_i} \sum_{j \in J_i} v_{ij} + \sum_{k \in J} v_k \in \sigma^\circ \cap \lambda_{\mathrm{lin}}.    
    $$
\end{proof}

\begin{proposition}\label{prop:fullDimInterMaxBigCone}
    Let $X:=X(A,P, \Sigma) \subseteq Z_X$ be an affine or complete explicit general arrangement variety and let $\sigma \in \Sigma$ be a maximal big cone, where we mean maximal in $\Sigma$ with respect to inclusion. Then we have
    $$
    \mathrm{dim}(\sigma \cap \lambda_\mathrm{lin}) = \mathrm{dim}(\lambda_\mathrm{lin}).
    $$
\end{proposition}
\begin{proof}
    If $X$ is affine, then $\Sigma$ consists of precisely one maximal big cone. As by construction the columns of $P$ generate $\QQ^{r+s}$ as a vector space, we conclude that $\sigma$ intersects $\lambda_{\mathrm{lin}}$ in full dimension.

    So assume $X$ is complete and let $\sigma \in \Sigma_X$ be a maximal big cone. Then due to Lemma \ref{lem:interiorIntersNotEmpty}, there exists a point $x \in \sigma^\circ \cap \lambda_\mathrm{lin}$.
    Moreover, as $X$ is complete, we have $|\mathrm{trop}(X)| \cap |\Sigma_X|= |\mathrm{trop}(X)|$, and therefore
    \begin{align*}
    \lambda _\mathrm{lin}= \bigcup_{\substack{\tau\in \Sigma\\\dim(\tau\cap \lambda_\mathrm{lin})
    =\dim(\lambda_\mathrm{lin})}} (\tau\cap \lambda_\mathrm{lin}).
\end{align*}
In particular there exists a cone $\tau \in \Sigma$ with $\dim(\tau \cap \lambda_\mathrm{lin})=\dim(\lambda_\mathrm{lin})$ and $x\in \sigma^\circ \cap\tau$.
As
$\sigma \cap \tau \preceq \sigma$ 
and 
$\sigma^\circ\cap\tau\neq\emptyset$ holds, we infer $\sigma \cap \tau=\sigma$ and thus $\sigma \preceq \tau$. Since $\sigma$ is a maximal cone, we conclude $\sigma =\tau$ and hence $\dim(\sigma\cap \lambda_\mathrm{lin})=\dim(\lambda_\mathrm{lin})$.
\end{proof}

\section{The Gorenstein Index via the Anticanonical Complex}
In this section we will describe how to read the Gorenstein index of an explicit general arrangement variety $X \subseteq Z_X$ off its anticanonical complex. We start by shortly recalling the construction of the anticanonical complex for these varieties and some basic facts on lattice distances.

\begin{construction}
Let $X(A, P, \Sigma) \subseteq Z_X$ be an explicit general arrangement variety. We consider the coarsest common refinement
$$\Sigma':= {\Sigma \sqcap \trop(X) := \left\{\sigma \cap \tau; \ \sigma \in \Sigma, \ \tau \in \trop(X)\right\}},$$
where $\mathrm{trop}(X)$ is endowed with the quasifan structure defined in Remark \ref{rem:tropAndMinimalAmbient}.
Let $\varphi \colon Z' \rightarrow Z$ be the toric morphism arising from the refinement of fans $\Sigma' \rightarrow \Sigma$ and let $X'$ be the proper transform of $X$ under $\varphi$. Then $Z'\rightarrow Z$ is called a \emph{weakly tropical resolution of $X$} and $X'\subseteq Z'$ fulfills the following conditions: 
\begin{enumerate}
    \item 
    $X' \subseteq Z'$ is again a general arrangement variety.
    \item The fan $\Sigma'$ consists of leaf cones.
    \item For any leaf cone $\sigma \in \Sigma$ we have $\sigma \in \Sigma'.$
\end{enumerate}

\end{construction}

Using the results of \cite{hw19} we obtain the following description of the anticanonical complex for general arrangement varieties:
\begin{construction}\label{constr:AKK}
Let $X \subseteq Z_X$ be an explicit $\QQ$-Gorenstein general arrangement variety and let $\varphi \colon Z' \rightarrow Z$ be its weakly tropical resolution.
For $0 \leq i \leq r$ we consider the following torus invariant divisors on $Z_X$:
$$
D_Z^{(i)} := \sum_{j=1}^{n_i} (r-c)l_{ij}D_{\varrho_{ij}} - \sum_{\varrho \in \Sigma^{(1)}} D_\varrho.
$$
Let $\sigma' \in \Sigma'$ be any cone. Then $\sigma'$ is a leaf cone and there exists an index $0 \leq i \leq r$ with $v_{ij} \notin \sigma'$ for all $1 \leq j \leq n_i$. Let $u_{\sigma'} \in M_\QQ$ be any element with $\mathrm{div}(\chi^{u_{\sigma'}}) = D_Z^{(i)}$. Then the anticanonical complex of $X \subseteq Z$ is given as 
$$
\mathcal{A}_X:= \bigcup_{\sigma' \in \Sigma'} A_{\sigma'}\qquad \quad
A_{\sigma'} := \sigma' \cap \left\{v \in N_\QQ; \ \bangle{u_{\sigma'}, v} \geq -1 \right\}.
$$
The \emph{relative interior} $\mathcal{A}_X^\circ$ of the anticanonical complex $\mathcal{A}_X$ is the interior of its support with respect to the tropical variety $\mathrm{trop}(X)$ 
and its \emph{boundary} is 
$$\partial \mathcal{A}_X := \mathcal{A}_X \setminus \mathcal{A}_X^\circ,$$
which we will assume to be endowed with the polyhedral complex structure
inherited from $\mathcal{A}_X.$
In particular, a cell of the anticanonical complex $\mathcal{A}_X$ lies in its boundary if and only if it does not contain $0$.
\end{construction}

\begin{example}\label{ex:AKKandGorenstein}
We consider the Fano explicit general arrangement variety $X := V(T_{01}^2T_{02} + T_{11}^2 + T_{21}^3)\subseteq \PP_{5,8,9,6} =:Z$ with Cox ring $R(A,P)$ defined by the matrices
$$A := \begin{bmatrix}
    -1 & 1 & 0\\-1 & 0 & 1
\end{bmatrix}
\quad \text{and} \quad
P=[v_{01},v_{02},v_{11},v_{21}] = \left[\begin{array}{cccc}
-2 & -1 & 2 & 0 
\\
 -2 & -1 & 0 & 3 
\\
 -1 & -2 & 1 & 2 
\end{array}\right].
$$
In particular, we have
$$\mathcal{R}(X) = R(A,P) = \CC[T_{01},T_{02},T_{11},T_{21}] / \bangle{T_{01}^2T_{02} + T_{11}^2 + T_{21}^3}$$
with generator degrees $[w_{01},w_{02},w_{11},w_{21}] = \begin{bmatrix}
    5& 8 & 9& 6
\end{bmatrix}$
and the fan $\Sigma$ corresponding to the minimal ambient toric variety $Z_X$ has the following three maximal cones: $\sigma_1 := \cone(v_{01},v_{11},v_{21})$, $\sigma_2 := \cone(v_{02},v_{11},v_{21})$ and $\sigma_3 := \cone(v_{01},v_{02})$.
The vertices of the anticanonical complex can then be calculated from these data using \cite[Cor. 6.5]{hw19}. These are $v_{01}, v_{02}, v_{11}, v_{21}$ and the points $u_1 = (0,0,2)$ and $u_2=(0,0,-1)$ in the lineality space $\lambda_{\mathrm{lin}}$ of the tropical variety $\mathrm{trop}(X)$.

We draw the anticanonical complex $\mathcal{A}_X$ and its boundary $\partial \mathcal{A}_X$ inside the tropical variety $\trop(X)$:
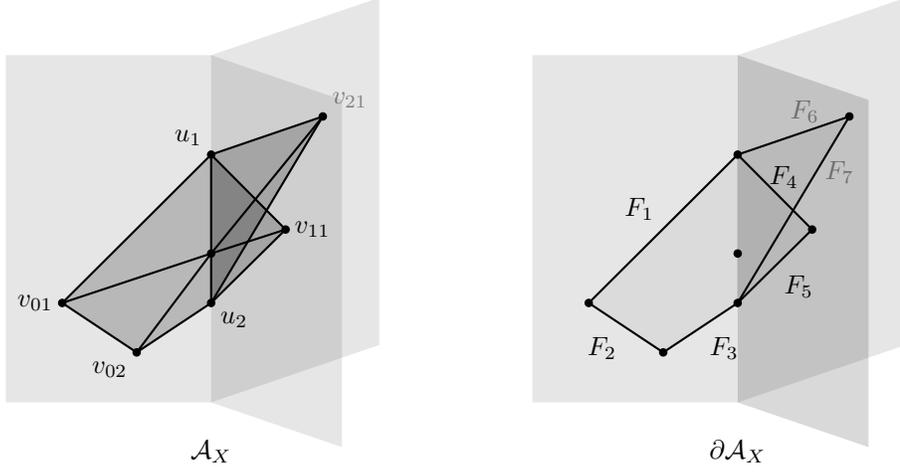
\begin{center}
\begin{tikzpicture}[tdplot_main_coords,
			edge/.style={thick},
			leaf0/.style={},
			leaf1/.style={},
			leaf2/.style={},
			leaf0e/.style={},
			leaf1e/.style={},
			leaf2e/.style={},
			le/.style={},
			axis/.style={},
			scale=.7]
	\coordinate (o) at (0,0,0);
	\coordinate (leaf0m) at (-\leafzero,-\leafzero,\linmin);
	\coordinate (leaf0p) at (-\leafzero,-\leafzero,\linmax);
	\coordinate (leaf0vm) at (\leafzero,\leafzero,\linmin);
	\coordinate (leaf0vp) at (\leafzero,\leafzero,\linmax);
	\coordinate (linm) at (0,0,\linmin);
	\coordinate (linp) at (0,0,\linmax);
	\coordinate (leaf1m) at (\leafone,0,\linmin);
	\coordinate (leaf1p) at (\leafone,0,\linmax);
	\coordinate (leaf1vm) at (-\leafone,0,\linmin);
	\coordinate (leaf1vp) at (-\leafone,0,\linmax);
	\coordinate (leaf2m) at (0,\leaftwo,\linmin);
	\coordinate (leaf2p) at (0,\leaftwo,\linmax);
	\coordinate (leaf2vm) at (0,-\leaftwo,\linmin);
	\coordinate (leaf2vp) at (0,-\leaftwo,\linmax);
	\coordinate (v01) at (-1,-1,-2);
	\coordinate (v02) at (-2,-2,-1);
	\coordinate (v1) at (2,0,1);
	\coordinate (v2) at (0,3,2);
	\coordinate (e1) at (0,0,2);
	\coordinate (e2) at (0,0,-1);
	
	
	\filldraw[leaf2,opacity=.1] (linp)--(leaf2p)--(leaf2m)--(linm);
	\draw[edge,leaf2e] (o) -- (v2);
	\draw[edge,leaf2e] (v2) -- (e1);
	\draw[edge,leaf2e] (v2) -- (e2);
	\filldraw[leaf2e,opacity=.2] (e1) -- (e2) -- (v2) -- cycle;

	\filldraw[leaf2e] (v2) circle (2pt) node[anchor=south west]{\color{gray}$v_{21}$};
	
	\filldraw[leaf1,opacity=.1] (linp)--(leaf1p)--(leaf1m)--(linm);
	\draw[edge,leaf1e] (o) -- (v1);
	\draw[edge,leaf1e] (v1) -- (e1);
	\draw[edge,leaf1e] (v1) -- (e2);
	\filldraw[leaf1e,opacity=.2] (e1) -- (e2) -- (v1) -- cycle;
	
	
	\filldraw[] (v1) circle (2pt) node[anchor=west]{$v_{11}$};  
	
	\filldraw[leaf0,opacity=.1] (linp)--(leaf0p)--(leaf0m)--(linm);
	
	\draw[edge,leaf0e] (o) -- (v01);
	\draw[edge,leaf0e] (o) -- (v02);
	\draw[edge,leaf0e] (v02) -- (e1);
	\draw[edge,leaf0e] (v01) -- (e2);
	\draw[edge,leaf0e] (v01) -- (v02);
	\filldraw[leaf0e,opacity=.2] (e1) -- (o) -- (v02) -- cycle;
	\filldraw[leaf0e,opacity=.2] (v01) -- (o) -- (v02) -- cycle;
	\filldraw[leaf0e,opacity=.2] (e2) -- (o) -- (v01) -- cycle;

	\filldraw[] (v01) circle (2pt) node[anchor=north east]{$v_{02}$};
	\filldraw[] (v02) circle (2pt) node[anchor=east]{$v_{01}$};
	
	\draw[edge] (o) -- (e1);
	\draw[edge] (o) -- (e2);
	
	
	\filldraw[] (o) circle (2pt);
	\filldraw[] (e1) circle (2pt) node[anchor=south east]{$u_1$};
	\filldraw[] (e2) circle (2pt) node[anchor=north west]{$u_2$};
	
	\draw[] ($(linm)+(0,0,-1)$) node{$\mathcal{A}_X$};

\begin{scope}[xshift=10cm]

	\coordinate (o) at (0,0,0);
	\coordinate (leaf0m) at (-\leafzero,-\leafzero,\linmin);
	\coordinate (leaf0p) at (-\leafzero,-\leafzero,\linmax);
	\coordinate (leaf0vm) at (\leafzero,\leafzero,\linmin);
	\coordinate (leaf0vp) at (\leafzero,\leafzero,\linmax);
	\coordinate (linm) at (0,0,\linmin);
	\coordinate (linp) at (0,0,\linmax);
	\coordinate (leaf1m) at (\leafone,0,\linmin);
	\coordinate (leaf1p) at (\leafone,0,\linmax);
	\coordinate (leaf1vm) at (-\leafone,0,\linmin);
	\coordinate (leaf1vp) at (-\leafone,0,\linmax);
	\coordinate (leaf2m) at (0,\leaftwo,\linmin);
	\coordinate (leaf2p) at (0,\leaftwo,\linmax);
	\coordinate (leaf2vm) at (0,-\leaftwo,\linmin);
	\coordinate (leaf2vp) at (0,-\leaftwo,\linmax);
	\coordinate (v01) at (-1,-1,-2);
	\coordinate (v02) at (-2,-2,-1);
	\coordinate (v1) at (2,0,1);
	\coordinate (v2) at (0,3,2);
	\coordinate (e1) at (0,0,2);
	\coordinate (e2) at (0,0,-1);
	
	
	\filldraw[leaf2,opacity=.1] (linp)--(leaf2p)--(leaf2m)--(linm);
	\draw[edge,leaf2e] (v2) -- (e1);
	\draw[edge,leaf2e] (v2) -- (e2);
	\filldraw[leaf2e,opacity=.05] (e1) -- (e2) -- (v2) -- cycle;
	
	\node[gray,anchor=south] (F6) at ($(v2)!0.4!(e1)$) {$F_6$};
	\node[gray,anchor=west] (F7) at ($(v2)!0.3!(e2)$) {$F_7$};
	
	\filldraw[leaf2e] (v2) circle (2pt);
	
	\filldraw[leaf1,opacity=.15] (linp)--(leaf1p)--(leaf1m)--(linm);
	\draw[edge,leaf1e] (v1) -- (e1);
	\draw[edge,leaf1e] (v1) -- (e2);
	\filldraw[leaf1e,opacity=.05] (e1) -- (e2) -- (v1) -- cycle;
	
	\node[leaf1e,anchor=west] (F4) at ($(v1)!0.7!(e1)$) {$F_4$};
	\node[leaf1e,anchor=north west] (F5) at ($(v1)!0.5!(e2)$) {$F_5$};
	
	\filldraw[] (v1) circle (2pt);  
	
	\filldraw[leaf0,opacity=.1] (linp)--(leaf0p)--(leaf0m)--(linm);

	\draw[edge,leaf0e] (v02) -- (e1);
	\draw[edge,leaf0e] (v01) -- (e2);
	\draw[edge,leaf0e] (v01) -- (v02);
	\filldraw[leaf0e,opacity=.05] (e1) -- (v02) -- (v01) -- (e2) -- cycle;
	
	\node[leaf0e,anchor=south east] (F1) at ($(v02)!0.5!(e1)$) {$F_1$};
	\node[leaf0e,anchor=north east] (F2) at ($(v02)!0.5!(v01)$) {$F_2$};
	\node[leaf0e,anchor=north west] (F2) at ($(v01)!0.5!(e2)$) {$F_3$};
	
	\filldraw[] (v02) circle (2pt);
	\filldraw[] (v01) circle (2pt); 
	
	
	\filldraw[] (o) circle (2pt);
	\filldraw[] (e1) circle (2pt);
	\filldraw[] (e2) circle (2pt);
	
	\draw[] ($(linm)+(0,0,-1)$) node{$\partial\mathcal{A}_X$};
\end{scope}
	
\end{tikzpicture}	
\end{center}
The maximal cells of $\partial\mathcal{A}_X$ are the line segments ${F_1 = \conv(u_1,v_{01})}$, ${F_2 = \conv(v_{01},v_{02})}$, $F_3 = \conv(u_2,v_{02})$, $F_4 = \conv(u_1,v_{11})$, $F_5 = \conv(u_2,v_{11})$, $F_6 = \conv(u_1,v_{21})$ and $F_7 = \conv(u_2,v_{21})$.

\end{example}

Now let us turn to lattices distances. A 
\emph{lattice subspace} is an affine subspace $A \subseteq \QQ^n$ such that $\mathrm{dim}(A) = \mathrm{rk}(A \cap \ZZ^n)$. 
Note that any affine subspace $A \subseteq \QQ^n$ that contains an element of $\ZZ^n$ is a lattice subspace. 
A \emph{lattice hyperplane} is a lattice subspace of codimension $1$. 

The \emph{lattice distance} $d(x,A)$ between a point $x \in \ZZ^n$ and a lattice subspace $A \subseteq \QQ^n$ is the number of lattice hyperplanes $H$ in the affine hull $\mathrm{aff}(A \cup \left\{x\right\})$ lying between $x$ and $A$, i.e.

$$
d(x,A) := \left\lvert\left\{H \subseteq \mathrm{aff}(A \cup \left\{x\right\}); \begin{array}{l}
H \text{ lattice hyperplane with } x \notin H \\\text{and } H \cap \mathrm{conv}(A \cup \left\{x\right\}) \neq \emptyset
\end{array}\right\}\right\rvert.
$$

It is well known that the lattice distance of a lattice hyperplane $H \subseteq \QQ^d$ and a point $x \in \ZZ^d$ can be calculated as follows: We have
$$
d(x, H) = |\bangle{u_H,v} - \bangle{u_H,x}|,
$$
where $u_H$ is a primitive normal of $H$ and $v$ is any point on $H$.
The lattice distance does not depend on unimodular transformations. For a convex set $B \subseteq \QQ^n$ with $\mathrm{aff}(B)$ a lattice subspace, we set
$d(x,B) := d(x, \mathrm{aff}(B)).$

Theorem \ref{thm:intro} is a direct consequence of the following proposition:
\begin{proposition}\label{prop:mainProp}
    Let $X \subseteq Z_X$ be an affine or complete explicit $\QQ$-Gorenstein general arrangement variety with anticanonical complex $\mathcal{A}_X$. Then the Gorenstein index $\iota_X$ of $X$ equals the least common multiple of the lattice distances of the maximal cells in the boundary of $\mathcal{A}_X$:
    $$
    \iota_X = \mathrm{lcm}(d(0, F); \ F \in \partial\mathcal{A}_X).
    $$
\end{proposition}
\begin{example}[Example \ref{ex:AKKandGorenstein} continued]
We calculate the Gorenstein index of the variety $X = V(T_{01}^2T_{02} + T_{11}^2 + T_{21}^3) \subseteq \PP_{5,8,9,6}$ as described in Example \ref{ex:AKKandGorenstein} 
using the above Proposition \ref{prop:mainProp}:
We have
$$
\begin{array}{cccc}
d(0, F_1) = 4, & d(0,F_2) = 3, & d(0, F_3) = 1, &d(0, F_4) = 4,\\
d(0,F_5) = 1,  & d(0, F_6) = 2, &d(0,F_7) = 1, &
\end{array}
$$
and obtain
$$
\iota_X = \mathrm{lcm}(d(0,F_i);  \ i = 1, \ldots, 7) = 12.
$$
\end{example}

\begin{lemma}\label{lem:GorensteinIndexViaDivCHiU}
    Let $X \subseteq Z_X$ be a $\QQ$-Gorenstein general arrangement variety. Then
    $$
    c_\sigma:= \min\! \left\{m \in \ZZ_{>0};  \begin{array}{l}
    \text{there exists } u \in \QQ^{r+s} \text{ with } m\!\cdot \! u \in \ZZ^{r+s} \\\text{and } \mathrm{div}(\chi^u)|_{Z_\sigma} = D_Z^{(i)}|_{Z_\sigma}
    \end{array}\right\}
    $$
    does not depend on the choice of $i \in \left\{0, \ldots, r\right\}$ and 
    the Gorenstein index $\iota_X$ of $X$ equals $\mathrm{lcm}(c_\sigma; \ \sigma \in \Sigma)$.
\end{lemma}
\begin{proof}
We recall that the pullback homomorphism $\mathrm{Cl}(Z_X) \rightarrow \mathrm{Cl}(X)$ is an isomorphism on the level of divisor class groups as well as on the level of Picard groups, see \cite{adhl}. In particular, as $X$ is $\QQ$-Gorenstein, each of the (linear equivalent) divisors $D_Z^{(i)}$ is $\QQ$-Cartier on $Z_X$ and their Cartier index equals the Gorenstein index of $X$.
As $Z_X$ is toric, for each $\sigma \in \Sigma$ we have
$$D_Z^{(i)}|_{Z_{\sigma}} = \mathrm{div}(\chi^u)|_{Z_\sigma}$$
for some $u \in \QQ^{r+s}$. Therefore, using $\mathrm{Cl}(Z_X) \cong \mathrm{Cl}(Z_X)^\TT$, we conclude that the Cartier index of $D_X^{(i)}$ on $Z_\sigma$ equals $c_\sigma$. In particular, $c_\sigma$ does not depend on the choice of $i$ and the Cartier index of $D_Z^{(i)}$ on $Z$ equals $\mathrm{lcm}(c_\sigma; \ \sigma \in \Sigma)$ as claimed.
\end{proof}

\begin{lemma}\label{lem:AKKPartsToWholeRoof}
Let $H \subseteq \QQ^{r+s}$ be a lattice hyperplane with $0 \notin H$.
Let $e_1, \ldots, e_{r+s}$ be the standard basis vectors and set $e_0:=-\sum e_i$ and consider for $0 \leq i \leq r$ the lattice subspaces
$$
H_i :=  H \cap \lambda_i, \quad \text{with } \lambda_i := \mathrm{cone}(e_i) + \mathrm{lin}(e_{r+1}, \ldots, e_{r+s}).
$$
If $\mathrm{dim}(H \cap \mathrm{lin}(e_{r+1}, \ldots, e_{r+s})) = s - 1$ and $\mathrm{dim}(H_i) = s$ holds for all $0 \leq i \leq r$, then for any subset $I \subseteq \left\{0, \ldots, r\right\}$ with $|I| = r$, we have
$$
d(0,H)  = \mathrm{lcm}(d(0,H_i); \ i \in I).
$$
\end{lemma}
\begin{proof}
We exemplarily prove the case $I = \left\{1, \ldots,r\right\}$. 
Let $b \in \QQ^{r+s}$ with $\bangle{b,v} = 1$ for all $v \in H$ and let $m \in \ZZ_{\geq 1}$ be the minimal element such that $m \cdot b \in \ZZ^{r+s}$. Then $m \cdot b$ is a primitive normal of $H$ and we have $d(0,H) = m$. Identifying $\mathrm{lin}(\lambda_i)$ with $\QQ ^{1+s}$ via the projection
$$
\pi_i \colon \QQ^{r+s} \rightarrow \QQ^{1+s}, \quad (a_1, \ldots, a_{r+s}) \mapsto (a_i, a_{r+1}, \ldots, a_{r+s})
$$
we can regard $H_i$ as a lattice hyperplane in $\QQ^{1+s}$ and $b^{(i)}:=(b_i, b_{r+1}, \ldots, b_{r+s})$ fulfills $\bangle{b^{(i)},v} =1$ for all $v\in H_i$. In particular, for the minimial $m^{(i)} \in \ZZ_{\geq 1}$ with $m^{(i)} \cdot b^{(i)} \in \ZZ^{1+s}$ we have $d(0, H_i) = m^{(i)}.$ Due to the structure of the $b^{(i)}$,  we conclude
$$
    d(0,H) = m = \mathrm{lcm}(m^{(i)}; \ 1 \leq i \leq r) = \mathrm{lcm}(d(0,H_i); \ 1 \leq i \leq r).
$$
\end{proof}
We will make frequent use of the following straightforward statement about lattice distances of lattice subspaces:
\begin{lemma}\label{lem:latticeDistanceMinimumLemma}
    Let $0 \notin A \subseteq M_\QQ$ be a lattices subspace. Then for every lattice subspace $0 \notin A'$ containing $A$ we have $d(0,A)\mid d(0,A')$ and in particular $d(0,A) \leq d(0,A')$. Moreover, we have
    $$
    d(0,A) = \mathrm{min}\!\left\{d(0,H); \ 0\notin H \subseteq M_\QQ \text{ lattice hyperplane with } A \subseteq H\right\}.
    $$
\end{lemma}
\begin{proof} 
    By replacing $M_\QQ$ with $\mathrm{lin}(A')$, it suffices to show the first assertion for lattice hyperplanes.
    Applying a suitable unimodular transformation we may futhermore assume $M = \ZZ^{n+m}$ and $\mathrm{aff}(A \cup \left\{0\right\})= \QQ^n \subseteq \QQ^{n+m}$. 
    In particular, there exists a unique primitive normal $u_A \in \ZZ^n$ of $A$ with $d(0,A) = \bangle{u_A,v}$ for any $v \in A$. Now let $0 \notin H \subseteq \ZZ^{n+m}$  be any hyperplane containing $A$. Then there is a primitive normal of $H$ of the form $u_H = (\lambda \cdot u_A, u)$ for some $u \in \ZZ^{n-s}$ and $\lambda \in \ZZ_{> 0}$. We conclude
    $\bangle{u_A,v} \mid \bangle{u_H, v}$ for any $v \in A \subseteq H$ and thus $d(0,A) \mid d(0,H)$. This shows the first assertion. 
    Moreover, the hyperplane $0 \notin H$ with $A\subseteq H$ and primitive normal $(u_A,0, \ldots, 0) \in \ZZ^n$ fulfills $d(0,A) = d(0,H)$ and we obtain the desired equality.    
\end{proof}

\begin{proof}[Proof of Proposition \ref{prop:mainProp}]
Let $\sigma \in \Sigma$ be any cone and let $u \in \QQ^{r+s}$ such that $\mathrm{div}(\chi^u)|_{Z_\sigma} = D_Z^{(i)}|_{Z_\sigma}$ holds. In a first step we show that the lattice distance $d(0, B_\sigma^{(i)})$ with
$$
B_\sigma^{(i)} :=  \sigma \cap \left\{v \in N_\QQ; \ \bangle{u, v} = -1\right\}
$$
equals $c_\sigma$ as defined in Lemma \ref{lem:GorensteinIndexViaDivCHiU}.
By construction, the hyperplanes $H$ with normal $u \in \QQ^{r+s}$ fulfilling $\mathrm{div}(\chi^u)|_{Z_\sigma} = D_Z^{(i)}|_{Z_\sigma}$ and $\bangle{u,v}  = -1$ for all $v \in H$ are precisely the hyperplanes containing $B_\sigma$.
Moreover, for these hyperplanes $H$ we have $d(0,H) = m$, where $m \in \ZZ_{>0}$ is the minimal integer with $m \cdot u \in \ZZ^{r+s}$.
Using Lemma \ref{lem:latticeDistanceMinimumLemma} we conclude $d(0,B_\sigma^{(i)}) = c_\sigma$ as claimed.

To complete the proof, we note that, in the notation of Construction \ref{constr:AKK}, the cells of the anticanonical complex $\mathcal{A}_X$ that lie in its boundary are the polyhedra
$$C_{\sigma'} := \sigma' \cap \left\{v \in N_\QQ; \ \bangle{u_{\sigma'}, v} = -1\right\}.$$
In particular, we are left with showing
that 
$$d(B_\sigma^{(i)}, 0) = \mathrm{lcm}(d(C_{\sigma'}, 0); \ \sigma' \in \Sigma' \text{ with } \sigma' \subseteq \sigma)$$
holds for every $\sigma' \in \Sigma'$.
As $C_{\sigma'} \subseteq B_\sigma^{(i)}$ for some $0 \leq i \leq r$ holds for all $\sigma'\subseteq \sigma$ and $d(0,B_\sigma^{(i)})$ does not depend on the choice of $i$, we obtain \ldq$\geq$\rdq\ using Lemma \ref{lem:latticeDistanceMinimumLemma}. For the inequality \ldq$\leq$\rdq\ we distinguish between the two types of cones occurring in $\Sigma$. So, let $\sigma \in \Sigma$ be a leaf cone. Then $\sigma$ is not affected by the weakly tropical resolution, that means we have $\sigma \in \Sigma'$.
We conclude $$d(0,B_\sigma^{(i)}) = d(0,C_{\sigma}) = \mathrm{lcm}(d(C_{\sigma'}, 0); \ \sigma' \in \Sigma' \text{ with } \sigma' \subseteq \sigma).$$ 
Now let $\sigma \in \Sigma$ be a big cone. As $d(0,B_\sigma^{(i)})$ does not depend on the choice of $i$, 
using Lemma \ref{lem:latticeDistanceMinimumLemma} it suffices to prove
$$
d(0,B_\sigma^{(i)}) = \mathrm{lcm}(d(0, C_{\sigma(j)}); \ \sigma(j) := \lambda_j \cap \sigma \text{ for } j \in \left\{0, \ldots, r\right\} \text{ with } j \neq i).
$$
for a maximal big cone $\sigma$.
In this situation, by construction of the anticanonical complex, see \ref{constr:AKK}, we have $C_{\sigma(j)} = B_\sigma^{(i)}\cap\lambda_j $. 
Using Proposition \ref{prop:fullDimInterMaxBigCone}, we obtain $\mathrm{dim}(B_\sigma^{(i)} \cap \lambda_\mathrm{lin}) = s-1$ and as $\sigma$ is a big cone we have $$\mathrm{dim}(C_{\sigma(j)}) = \mathrm{dim}(B_\sigma^{(i)}\cap\lambda_j) = s.$$ In particular we can apply Lemma \ref{lem:AKKPartsToWholeRoof} which proves the claim.
\end{proof}

\section{Applications}
In this section we apply our results to almost homogeneous Fano varieties $X$, where \emph{almost homogeneous} means that the automorphism group of $X$ has an open orbit in $X$. On the basis of the classification of all $\QQ$-factorial rational almost homogeneous Fano varieties with reductive automorphism group having a maximal torus of dimension two obtained in \cite[Prop. 8.6]{ahhl14}, we give concrete bounds on the defining data depending on the Gorenstein index, see Propositions \ref{prop:set1}, \ref{prop:set2}, \ref{prop:set3}, \ref{prop:set4} and \ref{prop:set5}. This enables us to filter the varieties for those of small Picard number, see Corollary \ref{cor:Classification} for the cases of Gorenstein index one, two and three.

Any almost homogeneous $\QQ$-factorial Fano threefold of Picard number one with reductive automorphism group having a maximal torus of dimension two is either the variety No.~\ref{arzhantsev-et-al-Prop-8-6-iv} from Corollary \ref{cor:Classification}, which is of Gorenstein index one, or arises up to isomorphy from one of the Settings \ref{set:1}, \ref{set:2}, \ref{set:3}, \ref{set:4} and \ref{set:5}, where we list the defining matrices $A$ and $P$, the fan $\Sigma$ of the minimal ambient toric variety $Z_X$ and the vertices of the anticanonical complex: 

\begin{setting}\label{set:1}
We have $A := \begin{bmatrix}
    -1 & 1 & 0\\-1 & 0 & 1
\end{bmatrix}$ and 
$$P=[v_{01},v_{02}, v_{11}, v_{12}, v_{21}] = 
\begin{bmatrix}
-1 & -1 & 1& 1& 0\\
-1& -1& 0& 0& l_{21}\\
-1 & 0& 0& 1& 0\\
0 &0& 0& d_{12}& d_{21}
\end{bmatrix}
$$
where $l_{21} > 1$, $d_{12} > 2$ and $- \frac{d_{21}}{d_{12}-1}< l_{21} < -d_{21}$ and the maximal cones of the fan $\Sigma$ corresponding to the minimal ambient toric variety are given as
\begin{align*}
\begin{array}{l@{\quad}l}
\sigma_1 := \cone(v_{01},v_{02},v_{11},v_{21}), & 
\sigma_2 := \cone(v_{01},v_{02},v_{12},v_{21}), \\
\sigma_3 := \cone(v_{01},v_{11},v_{12},v_{21}), &
\sigma_4 := \cone(v_{02},v_{11},v_{12},v_{21}),
\end{array}
\end{align*}
each of these is a big cone. Moreover, $\Sigma$ contains the four elementary big cones, $\sigma_1\cap \sigma_3$, $\sigma_1\cap \sigma_4$, $\sigma_2\cap \sigma_3$ and $\sigma_2\cap \sigma_4$. 
The vertices of the anticanonical complex can then be calculated from these data using \cite[Cor. 6.5]{hw19}. These are given as the columns of $P$ together with the following points in the lineality space of the tropical variety $\mathrm{trop}(X)$:
\begin{align*}
\begin{array}{l@{\quad}l}
v_{\sigma_1\cap \sigma_3}' =\left(0, 0, -\frac{l_{21}}{1+l_{21}}, \frac{d_{21}}{1+l_{21}}
\right), & v_{\sigma_1\cap \sigma_4}' =\left(0, 0, 0, \frac{d_{21}}{1+l_{21}}\right),\\[.5cm]
v_{\sigma_2\cap \sigma_3}' =\left(0, 0, 0, \frac{d_{12} l_{21}+d_{21}}{1+l_{21}}\right),& v_{\sigma_2\cap \sigma_4}' =\left(0, 0, 
\frac{l_{21}}{1+l_{21}}, \frac{d_{12} l_{21}+d_{21}}{1+l_{21}}\right)
\end{array}
\end{align*}
\end{setting}

\begin{proposition}\label{prop:set1}
Let $X$ be a Fano variety arising from Setting \ref{set:1} and denote by $\iota_X$ its Gorenstein index. Then 
we have $2<d_{12}\leq 3\iota_X$ and $-\iota_X \leq k< 0$ such that
$$
\left(k d_{12}+\iota_X \right) l_{21}+\iota_X \mid \iota_X k^{2} d_{12} \quad \text{and} \quad k d_{21} = \iota_X (l_{21}+1).   
$$
In particular, for fixed Gorenstein index there are finitely many varieties arising via this setting.
\end{proposition}
\begin{proof}
Due to the structure of the defining fan $\Sigma$ of $Z_X$ we obtain that $\conv(v_{21},v_{\sigma_2\cap\sigma_3}',v_{\sigma_2\cap\sigma_4}',0)$ is a cell in its anticanonical complex $\mathcal{A}_X$, and using Proposition \ref{prop:mainProp} we obtain
\begin{align*}
d(\mathrm{aff}(v_{21},&v_{\sigma_2\cap\sigma_3}',v_{\sigma_2\cap\sigma_4}'),0)\\ &= \lcm\left(\frac{d_{12} l_{21}+d_{21}}{\gcd(d_{12} l_{21}+d_{21},1+l_{21})}, \frac{d_{12} l_{21}+d_{21}}{\gcd(d_{12} l_{21}+d_{21},d_{12}-d_{21})}\right) \mid \iota_X.
\end{align*}
In particular, this implies $d_{12} l_{21}+d_{21}\mid \iota_X (1+l_{21})$ and thus
\begin{align}
d_{12} l_{21}+d_{21}\leq \iota_X (1+l_{21}).\label{eq:(i)ineq}
\end{align}

Similarly, since $\conv(v_{21},v_{\sigma_1\cap\sigma_3}',v_{\sigma_1\cap\sigma_4}',0)\in \mathcal{A}_X$, we see that
$$ d(\mathrm{aff}(v_{21},v_{\sigma_1\cap\sigma_3}',v_{\sigma_1\cap\sigma_4}'),0) = \frac{d_{21}}{\gcd(d_{21},l_{21}+1)} \mid \iota_X.$$
In particular, there exists some $k\in\ZZ$ such that 
\begin{align}
k d_{21} = \iota_X (l_{21}+1).\label{eq:(i)d21}
\end{align}
Note that because of $l_{21} < -d_{21}$, we have $-\iota_X \leq k < 0$. Inserting this into (\ref{eq:(i)ineq}) yields
\begin{align*}
d_{12}\leq \frac{\iota_X (1+l_{21}) - d_{21}}{l_{21}} \leq \frac{2\iota_X (1+l_{21})}{l_{21}} \leq 3\iota_X.
\end{align*}

We notice that $- \frac{d_{21}}{d_{12}-1}< l_{21} $ and the identity (\ref{eq:(i)d21}) ensure that $-kd_{12}\neq \iota_X$: If otherwise $-kd_{12} = \iota_X$, then $-\frac{d_{21}}{d_{12}-1} = -\frac{\iota_X (l_{21}+1)}{k (d_{12}-1)} = \frac{d_{12}}{d_{12}-1}(l_{21}+1)>l_{21}$.

Once again, we consider $d_{12} l_{21}+d_{21}\mid \iota_X (1+l_{21})$. Using (\ref{eq:(i)d21}) we infer
\begin{align*}
\left(k d_{12}+\iota_X \right) l_{21}+\iota_X \mid k \iota_X  l_{21}+k \iota_X
\end{align*}
Since $\left(k d_{12}+\iota_X \right)\neq 0$ as seen above, we have 
\begin{align*}
k \iota_X  l_{21}+k \iota_X &\mid (k \iota_X  l_{21}+k \iota_X)\left(k d_{12}+\iota_X \right).
\end{align*}
Therefore,
\begin{align*}
\left(k d_{12}+\iota_X \right) l_{21}+\iota_X \mid\   &(k \iota_X  l_{21}+k \iota_X)\left(k d_{12}+\iota_X \right) - k \iota_X \left(\left(k d_{12}+\iota_X \right) l_{21}+\iota_X\right) \\&= \iota_X k^{2} d_{12}.
\end{align*}
Thus, for fixed $\iota_X$ there are finitely many possibilities for $d_{12}$ and $k$. For each of these there are only finitely many possibilities for $l_{21}$ and thus also for $d_{21}$ by Equation~(\ref{eq:(i)d21}).
\end{proof}

\begin{setting}\label{set:2}
We have $A := \begin{bmatrix}
    -1 & 1 & 0\\-1 & 0 & 1
\end{bmatrix}$ and 
$$P=[v_{01}, v_{11}, v_{12}, v_{21}, v_{22}] = 
\begin{bmatrix}
-2 & 1& 1& 0& 0\\
-2 &0 &0& l_{21} &l_{22}\\
-1 &0& 1& 0 &0\\
d_{01} &0 &0& d_{21} &d_{22}
\end{bmatrix}
$$
where $l_{21}, l_{22} > 1, 2d_{22} > -d_{01}l_{22}, -2d_{21} > d_{01}l_{21}$. and the maximal cones of the fan $\Sigma$ corresponding to the minimal ambient toric variety are given as
\begin{align*}
\begin{array}{l@{\quad}l}
\sigma_1 := \cone(v_{01},v_{11},v_{12},v_{21}), & 
\sigma_2 := \cone(v_{01},v_{11},v_{12},v_{22}), \\
\sigma_3 := \cone(v_{01},v_{11},v_{21},v_{22}), &
\sigma_4 := \cone(v_{01},v_{12},v_{21},v_{22}),
\end{array}
\end{align*}
each of these is a big cone. Moreover, $\Sigma$ contains the four elementary big cones, $\sigma_1\cap \sigma_3$, $\sigma_1\cap \sigma_4$, $\sigma_2\cap \sigma_3$ and $\sigma_2\cap \sigma_4$. 
The vertices of the anticanonical complex can then be calculated from these data using \cite[Cor. 6.5]{hw19}. These are given as the columns of $P$ together with the following points in the lineality space of the tropical variety $\mathrm{trop}(X)$:
\begin{align*}
\begin{array}{l@{\quad}l}
v_{\sigma_1\cap \sigma_3}' = \left(0, 0, -\frac{l_{21}}{2+l_{21}}, 
\frac{d_{01} l_{21}+2 d_{21}}{2+l_{21}}\right),
& v_{\sigma_1\cap \sigma_4}' =\left(0, 0, 
\frac{l_{21}}{2+l_{21}}, 
\frac{d_{01} l_{21}+2 d_{21}}{2+l_{21}}\right) ,\\[.5cm] v_{\sigma_2\cap \sigma_3}' =\left(0, 0, 
-\frac{l_{22}}{2+l_{22}}, 
\frac{d_{01} l_{22}+2 d_{22}}{2+l_{22}}\right) , & v_{\sigma_2\cap \sigma_4}' =\left(0, 0, 
\frac{l_{22}}{2+l_{22}}, 
\frac{d_{01} l_{22}+2 d_{22}}{2+l_{22}}\right)
\end{array}
\end{align*}
\end{setting}

\begin{proposition}\label{prop:set2}
Let $X$ be a Fano variety arising from Setting \ref{set:2} and denote by $\iota_X$ its Gorenstein index. Then we end up in one of the following cases:
\begin{enumerate}
    \item $d_{01}=0$, $l_{21}=l_{22}\mid \iota_X$, $2\mid \iota_X$, $d_{21}\mid (2+l_{21})\frac{\iota_X}{2}$ and $d_{22}\mid (2+l_{22})\frac{\iota_X}{2}$.
    \item
    $d_{01}=0$, $l_{21}>l_{22}$, $2\mid \iota_X$, $1<l_{22}<\iota_X$, $0<d_{22}<\iota_X$ and $1\leq k <\iota_X$ such that
\begin{align*}
\frac{d_{21}}{\gcd(d_{21},d_{22})} \mid \frac{\iota_X}{2}+k \quad \text{and} \quad k(l_{21} d_{22}-d_{21} l_{22})=\iota_X(d_{22}-d_{21}),
\end{align*}
\item $d_{01}=-1$, $1<l_{21}<4\iota_X$, $0<s$, $s\mid \iota_X(l_{21}+2)$, $0<k\leq \iota_X$ and $0<t$ such that
\begin{align*}
t \mid 2\iota_X^2s+2ks\iota_X \quad \text{and} \quad k(tl_{21}+sl_{22})=2\iota_X(t+s),
\end{align*}
where $d_{21}=\frac{l_{21}-s}{2}$ and $d_{22}=\frac{l_{22}+t}{2}$; or the same with $(s,l_{21})$ and $(t,l_{22})$ interchanged.
\end{enumerate}
In particular, for fixed Gorenstein index there are finitely many varieties arising via this setting.
\end{proposition}

\begin{proof}
By suitable subtracting the first row from the last one, we can reach $d_{01}\in \{0,-1\}$. We start with $d_{01}=0$. In this case the conditions change to $l_{21}, l_{22} > 1, d_{22} > 0$ and $d_{21} < 0$. Let $l_{21}\geq l_{22}$ without loss of generality due to admissible operations. 

Due to the structure of the defining fan $\Sigma$ of $Z_X$ we obtain that $\conv(v_{21}, v_{22},v_{\sigma_1\cap\sigma_4}',v_{\sigma_2\cap\sigma_4}',0)$ is a cell in $\mathcal{A}_X$, and using Proposition \ref{prop:mainProp} we obtain 
\begin{align*}
d(\mathrm{aff}(v_{21},&v_{22},v_{\sigma_1\cap\sigma_4}',v_{\sigma_2\cap\sigma_4}'),0) = \frac{l_{22} d_{21}-d_{22} l_{21}}{\gcd(l_{22} d_{21}-d_{22} l_{21},d_{21}-d_{22})}=l_{21}=l_{22} \mid \iota_X. 
\end{align*}
Hence, $l_{21}$ and $l_{22}$ are bounded. Using again the structure of the defining fan $\Sigma$ of $Z_X$ we obtain that  $\conv(v_{01},v_{\sigma_1\cap\sigma_3}',v_{\sigma_1\cap\sigma_4}',0)$ and $\conv(v_{01},v_{\sigma_2\cap\sigma_3}',v_{\sigma_2\cap\sigma_4}',0)$ are cells of $\mathcal{A}_X$, and with Proposition \ref{prop:mainProp} we obtain 
\begin{align*}
d(\mathrm{aff}(v_{01},v_{\sigma_1\cap\sigma_3}',v_{\sigma_1\cap\sigma_4}'),0)  = \lcm\left(2, \frac{2d_{21}}{\gcd(2d_{21}, 2+l_{21})} \right) \mid \iota_X,\\
d(\mathrm{aff}(v_{01},v_{\sigma_2\cap\sigma_3}',v_{\sigma_2\cap\sigma_4}'),0)  = \lcm\left(2, \frac{2d_{22}}{\gcd(2d_{22}, 2+l_{22})} \right) \mid \iota_X.
\end{align*}
In particular, using $2\mid \iota_X$, this leads to the desired $d_{21}\mid (2+l_{21})\frac{\iota_X}{2}$ and $d_{22}\mid (2+l_{22})\frac{\iota_X}{2}$.

Secondly, let $l_{21}>l_{22}$. As above, we obtain
\begin{align*}
d(\mathrm{aff}(v_{21},&v_{22},v_{\sigma_1\cap\sigma_4}',v_{\sigma_2\cap\sigma_4}'),0) \\ &= \lcm\left(\frac{l_{21} d_{22}-d_{21} l_{22}}{\gcd(l_{21} d_{22}-d_{21} l_{22}, l_{21}-l_{22})}, \frac{l_{21} d_{22}-d_{21} l_{22}}{\gcd(l_{21} d_{22}-d_{21} l_{22},d_{22}-d_{21})} \right) \mid \iota_X. 
\end{align*}
In particular, this implies $l_{21} d_{22}-d_{21} l_{22} \mid \iota_X(l_{21}-l_{22})$ and $l_{21} d_{22}-d_{21} l_{22} \mid \iota_X(d_{22}-d_{21})$. Thus we have
\begin{align*}    
\frac{\iota_X(l_{21}-l_{22})}{l_{21} d_{22}-d_{21} l_{22}} \geq 1 \quad \text{and} \quad
\frac{\iota_X(d_{22}-d_{21})}{l_{21} d_{22}-d_{21} l_{22}} \geq 1,
\end{align*}
which yields
\begin{align*}
(d_{22}-d_{21})l_{22} \leq (\iota_X-d_{22})(l_{21}-l_{22}) \quad
\text{and} \quad  d_{22} (l_{21}-l_{22}) \leq (\iota_X-l_{22})(d_{22}-d_{21}).  
\end{align*}
 Since $(d_{22}-d_{21})l_{22} >0$ and $d_{22} (l_{21}-l_{22})>0$, we infer $d_{22}<\iota_X$ and $l_{22}<\iota_X$ by using $l_{21}-l_{22}>0$ and $d_{22}-d_{21}>0$. 

Once again, we consider $l_{21} d_{22}-d_{21} l_{22} \mid \iota_X(d_{22}-d_{21})$. In particular, there exists some $k\in \ZZ$ with $k\geq 1$ such that 
\begin{align}
    k(l_{21} d_{22}-d_{21} l_{22})=\iota_X(d_{22}-d_{21}). \label{eq0: (ii)}
\end{align}
Note that because of $l_{22}, l_{21}>1$, we have $l_{21} d_{22}-d_{21} l_{22}>d_{22}-d_{21}$ and thus $1\leq k <\iota_X$. Using (\ref{eq0: (ii)}) we obtain
\begin{align*}
(\iota_X-kl_{22})d_{21}=(\iota_X-kl_{21})d_{22}
\end{align*}
and thus $d_{21} \mid (\iota_X-kl_{21})d_{22}$. Due to the structure of the defining fan $\Sigma$ of $Z_X$ we obtain that  $\conv(v_{01},v_{\sigma_1\cap\sigma_3}',v_{\sigma_1\cap\sigma_4}',0)$ is a cell in $\mathcal{A}_X$, and using Proposition \ref{prop:mainProp} we obtain 
\begin{align*}
d(\mathrm{aff}(v_{01},v_{\sigma_1\cap\sigma_3}',v_{\sigma_1\cap\sigma_4}'),0)  = \lcm\left(2, \frac{2d_{21}}{\gcd(2d_{21}, 2+l_{21})} \right) \mid \iota_X.
\end{align*}
In particular, this implies $2d_{21}\mid \iota_X(2+l_{21})$ and $2\mid \iota_X$. Thus, we have $d_{21} \mid \iota_X+\frac{\iota_X}{2}l_{21}$. For $\tilde{d}_{21}:=\frac{d_{21}}{\gcd(d_{21},d_{22})}$ we infer $\tilde{d}_{21} \mid \iota_X+\frac{\iota_X}{2}l_{21}$ and $\tilde{d}_{21}\mid \iota_X -kl_{21}$, and thus
\begin{align*}
\tilde{d}_{21} \mid  \iota_X+\frac{\iota_X}{2}l_{21}-(\iota_X -kl_{21})=\left(\frac{\iota_X}{2}+k\right)l_{21}.
\end{align*}
Since $d_{21}$ and $l_{21}$ are coprime, we obtain $\tilde{d}_{21} \mid \frac{\iota_X}{2}+k$.
Thus, for fixed $\iota_X$ there are finitely many possibilities for $d_{22},l_{22}, \tilde{d}_{21}$ and $k$. For each of these there are only finitely many possibilities for $d_{21}$ and thus also for $l_{21}$ by (\ref{eq0: (ii)}).

We continue with $d_{01}=-1$. In this case the conditions change to $l_{21}, l_{22} > 1, -l_{22}+2d_{22}> 0$ and $l_{21}-2d_{21} > 0$. This yields $l_{21}d_{22}-l_{22}d_{21}>0$. We set 
\begin{align}
s:= l_{21}-2d_{21} \in \ZZ_{\geq 1} \quad \text{ and } \quad t:= -l_{22}+2d_{22} \in \ZZ_{\geq 1}. \label{s and t: (ii)}
\end{align}
Then we have
\begin{align*}
2(l_{21}d_{22}-l_{22}d_{21})=l_{21}(l_{22}+t)-l_{22}(l_{21}-s)=tl_{21}+sl_{22}.
\end{align*}
Due to the structure of the defining fan $\Sigma$ of $Z_X$ we obtain that $\conv(v_{01},v_{\sigma_1\cap\sigma_3}',v_{\sigma_2\cap\sigma_3}',0)$ is a cell in $\mathcal{A}_X$, and using Proposition \ref{prop:mainProp} we obtain 
\begin{align*}
d(\mathrm{aff}&(v_{01},v_{\sigma_1\cap\sigma_3}',v_{\sigma_2\cap\sigma_3}'),0)  \\ &= \lcm\left(\frac{d_{22}-d_{21}}{\gcd\left(l_{21}d_{22}-l_{22}d_{21},d_{22}-d_{21}\right)},\frac{l_{21}-l_{22}}{\gcd\left(l_{21}d_{22}-l_{22}d_{21},l_{21}-l_{22}\right)}\right) \mid \iota_X.
\end{align*}
This implies $l_{21}d_{22}-l_{22}d_{21} \mid \iota_X(l_{21}-l_{22})$ and $l_{21}d_{22}-l_{22}d_{21} \mid \iota_X(d_{22}-d_{21})$, and thus
\begin{align*}
l_{21}d_{22}-l_{22}d_{21} \mid \iota_X((l_{21}-l_{22})+2(d_{22}-d_{21}))=\iota_X(s+t).
\end{align*}
In particular, this yields
\begin{align}
tl_{21}+sl_{22}=2(l_{21}d_{22}-l_{22}d_{21}) \mid 2\iota_X(t+s), \label{tl_21+sl_22: (ii)}
\end{align}
and because of $l_{21}d_{22}-l_{22}d_{21}>0$ and $t+s>0$ we obtain $\frac{tl_{21}+sl_{22}}{t+s}\leq 2\iota_X$. 

We first consider $t\geq s$. Then we have 
\begin{align*}
\frac{l_{21}}{2}< \frac{tl_{21}+sl_{22}}{t+s} \leq 2\iota_X,
\end{align*}
and this yields $l_{21}<4\iota_X$. 
Due to the structure of the defining fan $\Sigma$ of $Z_X$ we obtain that 
$\conv(v_{11},v_{12}, v_{\sigma_1\cap\sigma_3}',v_{\sigma_1\cap\sigma_4}',0)$ is a cell in $\mathcal{A}_X$, and using Proposition \ref{prop:mainProp} we obtain 
\begin{align*}
d(\mathrm{aff}(v_{11},v_{12},v_{\sigma_1\cap\sigma_3}',v_{\sigma_1\cap\sigma_4}'),0)  = \lcm\left(\frac{s}{\gcd\left(s,l_{21}+2\right)}\right) \mid \iota_X.
\end{align*}
In particular, this implies $s\mid \iota_X(l_{21}+2)$, and $l_{21}<4\iota_X$ yields $s< 4\iota_X^2+2\iota_X$.
Similarly, since $\conv(v_{11},v_{12}, v_{\sigma_2\cap\sigma_3}',v_{\sigma_2\cap\sigma_4}',0) \in \mathcal{A}_X$ , we see that
\begin{align*}
d(\mathrm{aff}(v_{11},v_{12},v_{\sigma_2\cap\sigma_3}',v_{\sigma_2\cap\sigma_4}'),0)  = \lcm\left(\frac{t}{\gcd\left(t,l_{22}+2\right)}\right) \mid \iota_X.  
\end{align*}
In particular, this implies $t\mid \iota_X(l_{22}+2)$. Furthermore, due to \ref{tl_21+sl_22: (ii)} there exists some $k\in \ZZ$ such that
\begin{align}
k(tl_{21}+sl_{22})=2\iota_X(t+s). \label{eq-1: (ii)}
\end{align}
Note that because of $l_{21}, l_{22}>1$, we have $tl_{21}+sl_{22}\geq 2(s+t)$, and thus $0<k\leq \iota_X$. Using (\ref{eq-1: (ii)}) yields
\begin{align*}
 t(kl_{21}-2\iota_X)=2\iota_X s-ksl_{22},   
\end{align*}
so we get $t\mid 2\iota_X s-ksl_{22}$. Using $t\mid \iota_X(l_{22}+2)$ we obtain
\begin{align*}
t\mid \iota_X(2\iota_X s-ksl_{22}) +ks\iota_X(l_{22}+2)=2\iota_X^2s+2ks\iota_X.
\end{align*}
Thus, for fixed $\iota_X$ there are finitely many possibilities for $l_{21}, s,t$ and $k$. For each of these there are only finitely many possibilities for $l_{22}$ due to (\ref{eq-1: (ii)}).

For the case where $s >t$, one follows the same arguments as above with $(s,l_{21})$ and $(t,l_{22})$ interchanged.
In both cases, $d_{21}$ and $d_{22}$ are obtained by (\ref{s and t: (ii)}).
\end{proof}

\begin{setting}\label{set:3}
We have $A := \begin{bmatrix}
    -1 & 1 & 0\\-1 & 0 & 1
\end{bmatrix}$ and 
$$P=[v_{01}, v_{11}, v_{12}, v_{21}, v_{22}] = 
\begin{bmatrix}
-2 &1 &1 &0& 0\\
-2& 0& 0& 1& l_{22}\\
-1& 0& 1& 0& 0\\
d_{01}& 0& 0& d_{21}& d_{22}
\end{bmatrix}
$$
where $l_{22} > 1, d_{22} > d_{21}l_{22} + l_{22}, 2d_{22} > -d_{01}l_{22}, -2d_{21}> d_{01}$ and the maximal cones of the fan $\Sigma$ corresponding to the minimal ambient toric variety are given as
\begin{align*}
\begin{array}{l@{\quad}l}
\sigma_1 := \cone(v_{01},v_{11},v_{12},v_{21}), & 
\sigma_2 := \cone(v_{01},v_{11},v_{12},v_{22}), \\
\sigma_3 := \cone(v_{01},v_{11},v_{21},v_{22}), &
\sigma_4 := \cone(v_{01},v_{12},v_{21},v_{22}),
\end{array}
\end{align*}
each of these is a big cone. Moreover, $\Sigma$ contains the four elementary big cones, $\sigma_1\cap \sigma_3$, $\sigma_1\cap \sigma_4$, $\sigma_2\cap \sigma_3$ and $\sigma_2\cap \sigma_4$. 
The vertices of the anticanonical complex can then be calculated from these data using \cite[Cor. 6.5]{hw19}. These are given as the columns of $P$ together with the following points in the lineality space of the tropical variety $\mathrm{trop}(X)$:
\begin{align*}
\begin{array}{l@{\quad}l}
v_{\sigma_1\cap \sigma_3}' =\left(0, 0, -{\frac{1}{3}}, 
\frac{d_{01}}{3}+\frac{2 d_{21}}{3}\right),
& v_{\sigma_1\cap \sigma_4}' =\left(0, 0, {
\frac{1}{3}}, \frac{d_{01}}{3}+\frac{2 d_{21}}{3}\right),\\[.5cm] v_{\sigma_2\cap \sigma_3}' =\left(0, 0, 
-\frac{l_{22}}{2+l_{22}}, 
\frac{d_{01} l_{22}+2 d_{22}}{2+l_{22}}\right), & v_{\sigma_2\cap \sigma_4}' =\left(0, 0, \frac{l_{22}}{2+l_{22}}, 
\frac{d_{01} l_{22}+2 d_{22}}{2+l_{22}}\right)
\end{array}
\end{align*}
\end{setting}

\begin{proposition}\label{prop:set3}
Let $X$ be a Fano variety arising from Setting \ref{set:2} and denote by $\iota_X$ its Gorenstein index. Then we have $d_{21}=0$, $-3\iota_X \leq d_{01}<0$, $-3\iota_X \leq k_{01}<0$ and $0 < k_{22}<\iota_X$ such that
\begin{align*}
\iota_X\left( \frac{3}{k_{01}}+\frac{2}{k_{22}}\right)l_{22}-\frac{2\iota_X}{k_{22}} \mid 6 \iota_X(k_{22}+k_{01}) \quad \text{and} \quad d_{22} k_{22} = \iota_X (l_{22}-1).
\end{align*}
In particular, for fixed Gorenstein index there are finitely many varieties arising via this setting.
\end{proposition}

\begin{proof}
By subtracting $d_{21}$ times the second row from the last one, we can reach $d_{21}=0$. The conditions change to $l_{22} > 1, d_{22} > l_{22}, 2d_{22} > -d_{01}l_{22}$ and $0> d_{01}$.
Due to the structure of the defining fan $\Sigma$ of $Z_X$ we obtain that $\conv(v_{01},v_{\sigma_1\cap\sigma_3}',v_{\sigma_1\cap\sigma_4}',0)$ is a cell in $\mathcal{A}_X$, and using Proposition \ref{prop:mainProp} we obtain
\begin{align*}
d(\mathrm{aff}(v_{01},&v_{\sigma_1\cap\sigma_3}',v_{\sigma_1\cap\sigma_4}'),0)= \frac{d_{01}}{\gcd\left(d_{01},3\right)}\mid \iota_X.
\end{align*}
In particular, this implies $d_{01}\mid 3\iota_X$ and thus there exists some $k_{01} \in \ZZ$ with 
\begin{align}
d_{01} k_{01}=3\iota_X.\label{eq:(iii) part 1 d01}
\end{align}
Note that because of $-3\iota_X \leq d_{01}<0$, we have $-3\iota_X \leq k_{01} <0$. Similarly, since $\conv(v_{21},v_{22},v_{\sigma_1\cap\sigma_4}',
v_{\sigma_2\cap\sigma_4}',0)\in \mathcal{A}_X$, we see that
\begin{align*}
d(\mathrm{aff}(v_{21},v_{22},&v_{\sigma_1\cap\sigma_4}',v_{\sigma_2\cap\sigma_4}'),0) = \frac{d_{22}}{\gcd(l_{22}-1,d_{22})}  \mid \iota_X.
\end{align*}
In particular, there exists some $k_{22}\in\ZZ$ such that 
\begin{align}
d_{22} k_{22} = \iota_X (l_{22}-1).\label{eq:(iii) part 1 d22}
\end{align}
Note that because of $1<l_{22} < d_{22}$, we have $0 < k_{22} < \iota_X$. Similarly, since $\conv(v_{01},v_{\sigma_2\cap\sigma_3}',
v_{\sigma_2\cap\sigma_4}',0)\in \mathcal{A}_X$, we see that
\begin{align*}
d(\mathrm{aff}(v_{01},&v_{\sigma_2\cap\sigma_3}',v_{\sigma_2\cap\sigma_4}'),0) \\ &= \lcm\left(\frac{d_{01} l_{22}+2 d_{22}}{\gcd(-d_{01}+d_{22},d_{01} l_{22}+2 d_{22})}, \frac{d_{01} l_{22}+2 d_{22}}{\gcd(2+l_{22},d_{01} l_{22}+2 d_{22})}\right) \mid \iota_X.
\end{align*}
In particular, this implies  $d_{01} l_{22}+2 d_{22} \mid \iota_X(2+l_{22})$. Using (\ref{eq:(iii) part 1 d01}) and (\ref{eq:(iii) part 1 d22}) we obtain
\begin{align*}
d_{01} l_{22}+2 d_{22}= \underbrace{\iota_X\left( \frac{3}{k_{01}}+\frac{2}{k_{22}}\right)l_{22}-\frac{2\iota_X}{k_{22}}}_{=: b} \mid \iota_X(2+l_{22}).
\end{align*}
Assuming $\frac{3}{k_{01}}+\frac{2}{k_{22}}=0$, then we get
\begin{align*}
d_{01} l_{22}+2 d_{22}= -\frac{2\iota_X }{k_{22}}<0,
\end{align*}
that contradicts $d_{01} l_{22}+2 d_{22}>0$. Thus we have 
\begin{align}
0 \neq \frac{3}{k_{01}}+\frac{2}{k_{22}}=\frac{3k_{22}+2k_{01}}{k_{22} k_{01}}. \label{eq: (iii) part 1 neq 0}
\end{align}
Once again, we consider $d_{01} l_{22}+2 d_{22} \mid \iota_X(2+l_{22})$. By using (\ref{eq: (iii) part 1 neq 0}) we infer
\begin{align}
b \mid \ \iota_X l_{22}+2\iota_X \mid (3k_{22}+2k_{01})(\iota_X l_{22}+2\iota_X). \label{eq: (iii) part 1 final}
\end{align}
This implies
\begin{align*}
b \mid (3k_{22}+2k_{01})(\iota_X l_{22}+2\iota_X)-k_{22} k_{01} b=6 \iota_X(k_{22}+k_{01}).
\end{align*}
Assuming $k_{22}=-k_{01}$, then we get
\begin{align*}
d_{01} l_{22}+2 d_{22}= \iota_X \left( \frac{3}{k_{01}}-\frac{2}{k_{01}}\right)l_{22}+\frac{2\iota_X}{k_{01}}=\frac{\iota_X(l_{22}+2)}{k_{01}}<0,
\end{align*}
that contradicts $d_{01} l_{22}+2 d_{22}>0$.
Thus, for fixed $\iota_X$ there are finitely many possibilities for $d_{01},k_{01}$ and $k_{22}$. For each of these there are only finitely many possibilities for $l_{22}$ and thus also for $d_{22}$ by (\ref{eq: (iii) part 1 final}) and (\ref{eq:(iii) part 1 d22}).
\end{proof}

\begin{setting}\label{set:4}
We have $A := \begin{bmatrix}
    -1 & 1 & 0\\-1 & 0 & 1
\end{bmatrix}$ and 
$$P=[v_{01}, v_{11}, v_{12}, v_{21}, v_{22}] = 
\begin{bmatrix}
-2 &1 &1 &0& 0\\
-2& 0& 0& 1& l_{22}\\
-1& 0& 1& 0& 0\\
d_{01}& 0& 0& d_{21}& d_{22}
\end{bmatrix}
$$
with $l_{22} > 1, 2d_{22} > -d_{01}l_{22}, 1 - 2d_{21} > d_{01}$.
There are at least three maximal cones in the fan $\Sigma$ corresponding to the minimal ambient toric variety, which are given as
\begin{align*}
\begin{array}{l@{\quad}l}
\sigma_1 := \cone(v_{01},v_{11},v_{12},v_{22}), & 
\sigma_2 := \cone(v_{01},v_{11},v_{21},v_{22}), \\
\sigma_3 := \cone(v_{01},v_{12},v_{21},v_{22}).
\end{array}
\end{align*}
Each of these is a big cone. If $2d_{21}+d_{01}\neq 0$ holds, then there exists a fourth maximal cone of $\Sigma$ that is a big cone, namely $\sigma_4 := \cone(v_{01},v_{11},v_{12},v_{21})$.
Moreover, $\Sigma$ contains the four elementary big cones, $\sigma_1\cap \sigma_2$, $\sigma_1\cap \sigma_3$,$\tau_1 \preceq \sigma_2$, $\tau_2 \preceq \sigma_3$. 
The vertices of the anticanonical complex can then be calculated from these data using \cite[Cor. 6.5]{hw19}. These are given as the columns of $P$ together with the following points in the lineality space of the tropical variety $\mathrm{trop}(X)$:
\begin{align*}
\begin{array}{l@{\quad}l}
v_{\tau_1}' =\left(0, 0, -{\frac{1}{3}}, 
\frac{d_{01}}{3}+\frac{2 d_{21}}{3}\right), & v_{\tau_2}' =\left(0, 0, {
\frac{1}{3}}, \frac{d_{01}}{3}+\frac{2 d_{21}}{3}\right)\\[.5cm] v_{\sigma_1\cap \sigma_2}' =\left(0, 0, 
-\frac{l_{22}}{2+l_{22}}, 
\frac{d_{01} l_{22}+2 d_{22}}{2+l_{22}}\right),
& v_{\sigma_1\cap \sigma_3}' =\left(0, 0, \frac{l_{22}}{2+l_{22}}, 
\frac{d_{01} l_{22}+2 d_{22}}{2+l_{22}}\right)
\end{array}
\end{align*}
\end{setting}

\begin{proposition}\label{prop:set4}
Let $X$ be a Fano variety arising from Setting \ref{set:4} and denote by $\iota_X$ its Gorenstein index. Then we have
$d_{21}=0$, $-2\iota_X \leq d_{01}\leq 0$ and $0<k<2\iota_X$ such that
\begin{align*}
 \left(d_{01} k+2 \iota_X \right) \frac{ l_{22}}{k}-\frac{2 \iota_X}{k} \mid 2\iota_X(d_{01}k+3\iota_X) \quad \text{and} \quad   d_{22} k=\iota_X (l_{22}-1).
\end{align*}
In particular, for fixed Gorenstein index there are finitely many varieties arising via this setting.
\end{proposition}

\begin{proof}
By subtracting $d_{21}$ times the second row from the last one, we can reach $d_{21}=0$. The conditions change to $l_{22} > 1, 2d_{22} > -d_{01}l_{22}$ and $1  > d_{01}$. Due to the structure of the defining fan $\Sigma$ of $Z_X$ we obtain that
$\conv(v_{21},v_{22},v_{\tau_1}',v_{\sigma_1 \cap \sigma_2}',0)$ is a cell in $\mathcal{A}_X$, and using Proposition \ref{prop:mainProp} we obtain
\begin{align*}
d(\mathrm{aff}(v_{21},v_{22},&v_{\tau_1}',v_{\sigma_1\cap\sigma_2}'),0) = \frac{d_{22}}{\gcd(l_{22}-1,d_{22})} \mid \iota_X. 
\end{align*}
In particular, this implies $d_{22} \mid \iota_X(l_{22}-1)$. Because of $d_{22}>0$ and $l_{22}>1$, we have
\begin{align*}
\frac{d_{22}}{l_{22}-1}<\iota_X.
\end{align*}
Using $2d_{22}>-d_{01}l_{22}$ we infer
\begin{align*}
-\frac{d_{01}}{2}\cdot \frac{l_{22}}{l_{22}-1}<\iota_X,
\end{align*}
that implies $-2\iota_X \leq d_{01} \leq 0$. If $d_{01}=0$, then the last coordinates of $v_{\tau_1}', v_{\tau_2}', v_{\sigma_1\cap \sigma_2}'$ and $v_{\sigma_1\cap \sigma_3}'$ would all be non-negative, which contradicts $0 \in \vert \mathcal{A}_X\vert ^\circ$. Thus we have $d_{01}<0$ and $d_{22}>\frac{l_{22}}{2}$. 
Once again, we consider $d_{22} \mid \iota_X(l_{22}-1)$. Thus there exists $k \in \ZZ$ with 
\begin{align}
d_{22} k=\iota_X (l_{22}-1) \quad \Leftrightarrow \quad d_{22}=\frac{\iota_X(l_{22}-1)}{k}. \label{eq: (iii) part 2 d_22}
\end{align}
Note that because of $d_{22}>\frac{l_{22}}{2}$ we have $0<k<2\iota_X$. Since $\conv(v_{11},v_{12},v_{\sigma_1 \cap \sigma_2}',v_{\sigma_1 \cap \sigma_3}',0)$ is a cell in $\mathcal{A}_X$, and using Proposition \ref{prop:mainProp} we obtain
\begin{align*}
d(\mathrm{aff}(v_{11},v_{12},&v_{\sigma_1\cap \sigma_2}',v_{\sigma_1\cap\sigma_3}'),0) = \frac{d_{01} l_{22}+2 d_{22}}{\gcd(2+l_{22},d_{01} l_{22}+2 d_{22})} \mid \iota_X. 
\end{align*}
In particular, this implies $d_{01} l_{22}+2 d_{22} \mid \iota_X (2+l_{22})$. Inserting (\ref{eq: (iii) part 2 d_22}) into this yields
\begin{align*}
d_{01} l_{22}+2 d_{22}=\underbrace{\left(d_{01} k+2 \iota_X \right) \frac{ l_{22}}{k}-\frac{2 \iota_X}{k}}_{=:b} \mid \iota_X(2+l_{22}).
\end{align*}
Assuming $d_{01}k+2\iota_X=0$, then we get $2\iota_X=-d_{01}k$. With $-2\iota_X \leq d_{01}<0$ and $0<k<2\iota_X$ we infer $2\leq -d_{01}, k<\iota_X$. This contradicts $2d_{22} > -d_{01}l_{22}$, thus we have $d_{01}k+2\iota_X\neq 0$. Using this we get
\begin{align*}
b \mid 2\iota_X+\iota_X l_{22} \mid (d_{01}k+2\iota_X)(2\iota_X+\iota_X l_{22}).
\end{align*}
This implies 
\begin{align}
b \mid (d_{01}k+2\iota_X)(2\iota_X+\iota_X l_{22})-k\iota_X b=2\iota_X(d_{01}k+3\iota_X). \label{eq: (iii) part 2 final}
\end{align} 
Assuming $d_{01}=-\frac{3\iota_X}{k}$, then we get
\begin{align*}
d_{01} l_{22}+2 d_{22}= (-3\iota_X+2\iota_X)\frac{l_{22}}{k}-\frac{2\iota_X}{k}=-\frac{\iota_X(l_{22}+2)}{k}<0,
\end{align*}
that contradicts $d_{01} l_{22}+2 d_{22}>0$.
Thus, for fixed $\iota_X$ there are finitely many possibilities for $d_{01}$ and $k$. For each of these there are only finitely many possibilities for $l_{22}$ and thus also for $d_{22}$ by (\ref{eq: (iii) part 2 final}).
\end{proof}

\begin{setting}\label{set:5}
We have $A := \begin{bmatrix}
    -1 & 1 & 0\\-1 & 0 & 1
\end{bmatrix}$ and 
$$P=[v_{01},v_{11},v_{12},v_{21},v_1]= 
\begin{bmatrix}
-2 &1& 1& 0& 0\\
-2& 0& 0& l_{21}& 0\\
-1& 0& 1& 0& 0\\
1& 0& 0& d_{21}& 1
\end{bmatrix}  $$ with $1 < l_{21} < -2d_{21} < 2l_{21}$ and the maximal cones of the $\Sigma$ corresponding to the minimal ambient toric variety are given as
\begin{align*}
\begin{array}{l@{\quad}l}
\sigma_1 := \cone(v_{01},v_{11},v_{21},v_{1}), & \sigma_2 := \cone(v_{01},v_{12},v_{21},v_{1}), \\
\sigma_3 := \cone(v_{01},v_{11},v_{12},v_{21}), &
\sigma_4 := \cone(v_{11},v_{12},v_{1}),
\end{array}
\end{align*}
where $\sigma_1,\sigma_2,\sigma_3$ are big cones and $\sigma_4$ is a leaf cone. Moreover, $\Sigma$ contains two elementary big cones, $\sigma_1\cap \sigma_3$ and $\sigma_2\cap \sigma_3$. 
The vertices of the anticanonical complex can then be calculated from these data using \cite[Cor. 6.5]{hw19}. These are given as the columns of $P$ together with the following points in the lineality space of the tropical variety $\mathrm{trop}(X)$:
\begin{align*}
v_{\sigma_1\cap \sigma_3}' =\left(0, 0, -\frac{l_{21}}{2+l_{21}}, 
\frac{l_{21}+2 d_{21}}{2+l_{21}}\right)
 \quad \text{and}\quad v_{\sigma_2\cap \sigma_3}' = \left(0, 0, 
 \frac{l_{21}}{2+l_{21}}, \frac{l_{21}+2 d_{21}}{2+l_{21}}\right).
\end{align*}
\end{setting}

\begin{proposition}\label{prop:set5}
Let $X$ be a Fano variety arising from Setting \ref{set:4} and denote by $\iota_X$ its Gorenstein index. Then we have $-\iota_X \leq k<-\frac{\iota_X}{2}$ such that
\begin{align*}
l_{21}\left(\frac{2k+\iota_X}{\iota_X}\right) + 2\mid  4k\iota_X \quad \text{and} \quad \frac{k l_{21}}{\iota_X} = d_{21}-1.
\end{align*}
In particular, for fixed Gorenstein index there are finitely many varieties arising via this setting.
\end{proposition}

\begin{proof}
Due to the structure of the defining fan $\Sigma$ of $Z_X$ we obtain that $\conv(v_{21},v_1,v_{\sigma_1\cap\sigma_3}',0)$ is a cell in $\mathcal{A}_X$, and using Proposition \ref{prop:mainProp} we obtain $$ d(\mathrm{aff}(v_{21},v_1,v_{\sigma_1\cap\sigma_3}'),0) = \frac{l_{21}}{\gcd(l_{21},d_{21}-1)} \mid \iota_X.$$
In particular, this implies
\begin{align}
\frac{k l_{21}}{\iota_X} = d_{21}-1\label{eq:(v)d21}
\end{align}
for some $k\in\ZZ$. Because of $l_{21} < -2d_{21} < 2l_{21}$, we have $-\iota_X \leq k <-\frac{\iota_X}{2}$.
Similarly, since $\conv(v_{11},v_{12},v_{\sigma_1\cap\sigma_3}',v_{\sigma_2\cap\sigma_3}',0)\in \mathcal{A}_X$, we see that
$$ d(\mathrm{aff}(v_{11},v_{12},v_{\sigma_1\cap\sigma_3}',v_{\sigma_2\cap\sigma_3}'),0) = \frac{l_{21}+2 d_{21}}{\gcd(l_{21}+2 d_{21},l_{21}+2)} \mid \iota_X.$$
In particular, we obtain $(l_{21}+2 d_{21})\mid \iota_X(l_{21}+2)$. Using (\ref{eq:(v)d21}), we infer
\begin{align*}
\left(l_{21} + \frac{2k}{\iota_X}l_{21}+2\right)= \left(l_{21}\left(\frac{2k+\iota_X}{\iota_X}\right) + 2\right) \mid \iota_X(l_{21}+2)
\end{align*}
We notice that $-\iota_X\leq 2k+\iota_X < 0$. Therefore, $\iota_X(l_{21}+2)\mid (2k+\iota_X)\iota_X(l_{21}+2)$. Hence,
\begin{align*}
\left(l_{21}\left(\frac{2k+\iota_X}{\iota_X}\right) + 2\right) \mid (2k+\iota_X)\iota_X(l_{21}+2) - \iota_X^2\left(l_{21}\left(\frac{2k+\iota_X}{\iota_X}\right) + 2\right) = 4k\iota_X.
\end{align*}
Thus, for fixed $\iota_X$ there are finitely many possibilities for $k$ and thus finitely many possibilities for $l_{21}$.
\end{proof}

\bibliography{bibliography}

\begin{bibdiv}
\begin{biblist}

\bib{abhw}{article}{
      author={Arzhantsev, Ivan},
      author={Braun, Lukas},
      author={Hausen, Jürgen},
      author={Wrobel, Milena},
       title={Log terminal singularities, platonic tuples and iteration of cox rings},
        date={2017oct},
     journal={European Journal of Mathematics},
      volume={4},
      number={1},
       pages={242\ndash 312},
         url={https://doi.org/10.1007/s40879-017-0179-8},
}

\bib{adhl}{book}{
      author={Arzhantsev, Ivan},
      author={Derenthal, Ulrich},
      author={Hausen, Jürgen},
      author={Laface, Antonio},
       title={Cox rings},
   publisher={Cambridge University Press},
     address={New York},
        date={2015},
}

\bib{ahhl14}{article}{
      author={Arzhantsev, Ivan},
      author={Hausen, Jürgen},
      author={Herppich, Elaine},
      author={Liendo, Alvaro},
       title={{The automorphism group of a variety with torus action of complexity one}},
        date={2014},
     journal={Moscow mathematical journal},
      volume={14},
      number={3},
       pages={429\ndash 471},
}

\bib{Ath}{article}{
      author={Athanasiadis, Christos~A.},
       title={Ehrhart polynomials, simplicial polytopes, magic squares and a conjecture of stanley},
        date={2005},
     journal={Journal für die reine und angewandte Mathematik},
      volume={2005},
      number={583},
       pages={163\ndash 174},
         url={https://doi.org/10.1515/crll.2005.2005.583.163},
}

\bib{BatBo}{article}{
      author={Batyrev, Victor},
      author={Borisov, Lev},
       title={Dual cones and mirror symmetry for generalized calabi-yau manifolds},
        date={1994},
     journal={arXiv: Algebraic Geometry},
         url={https://api.semanticscholar.org/CorpusID:15412062},
}

\bib{BaNill}{article}{
      author={Batyrev, Victor},
      author={Nill, Benjamin},
       title={h-vectors of gorenstein polytopes},
        date={2008},
     journal={Combinatorial aspects of mirror symmetry, Integer points in polyhedra—geometry, number theory, representation theory, algebra, optimization, statistics, Contemp. Math},
      volume={452},
       pages={35\ndash 66},
}

\bib{bhhn}{article}{
      author={Bechtold, Benjamin},
      author={Hausen, Jürgen},
      author={Huggenberger, Elaine},
      author={Nicolussi, Michele},
       title={{On Terminal Fano 3-Folds with 2-Torus Action}},
        date={201606},
     journal={International Mathematics Research Notices},
      volume={2016},
      number={5},
       pages={1563\ndash 1602},
}

\bib{bh}{article}{
      author={Braun, Lukas},
      author={H{\"a}ttig, Daniel},
       title={{Canonical threefold singularities with a torus action of complexity one and $k$-empty polytopes}},
        date={2020},
     journal={Rocky Mountain Journal of Mathematics},
      volume={50},
      number={3},
       pages={881 \ndash  939},
         url={https://doi.org/10.1216/rmj.2020.50.881},
}

\bib{BrunsRoemer}{article}{
      author={Bruns, Winfried},
      author={Römer, Tim},
       title={h-vectors of gorenstein polytopes},
        date={2007},
        ISSN={0097-3165},
     journal={Journal of Combinatorial Theory, Series A},
      volume={114},
      number={1},
       pages={65\ndash 76},
         url={https://www.sciencedirect.com/science/article/pii/S0097316506000471},
}

\bib{con}{article}{
      author={Conrads, Heinke},
       title={Weighted projective spaces and reflexive simplices},
        date={2002},
     journal={Manuscripta Math.},
       pages={1432\ndash 1785},
}

\bib{cls}{book}{
      author={Cox, David~A.},
      author={Little, John~B.},
      author={Schenck, Henry~K.},
       title={Toric varieties},
   publisher={American Mathematical Society},
     address={Providence, RI},
        date={2010},
}

\bib{ful}{book}{
      author={Fulton, William},
       title={Introduction to toric varieties. (am-131)},
   publisher={Princeton University Press},
        date={1993},
        ISBN={9780691000497},
         url={http://www.jstor.org/stable/j.ctt1b7x7vc},
}

\bib{hh}{article}{
      author={Hausen, Jürgen},
      author={Herppich, Elaine},
       title={Factorially graded rings of complexity one},
        date={2013},
       pages={414–428},
}

\bib{hhs}{article}{
      author={Hausen, Jürgen},
      author={Herppich, Elaine},
      author={Süss, Hendrik},
       title={Multigraded factorial rings and fano varieties with torus action.},
    language={eng},
        date={2011},
     journal={Documenta Mathematica},
      volume={16},
       pages={71\ndash 109},
         url={http://eudml.org/doc/232787},
}

\bib{wrobel:ta_hc}{article}{
      author={Hausen, Jürgen},
      author={Hische, Christoff},
      author={Wrobel, Milena},
       title={On torus actions of higher complexity},
        date={2019},
     journal={Forum of Mathematics, Sigma},
      volume={7},
}

\bib{hmw}{article}{
      author={Hausen, Jürgen},
      author={Mauz, Christian},
      author={Wrobel, Milena},
       title={The anticanonical complex for non-degenerate toric complete intersections},
        date={2020},
     journal={Manuscripta Math.},
       pages={1432\ndash 1785},
}

\bib{hs}{article}{
      author={Hausen, Jürgen},
      author={Süß, Hendrik},
       title={The cox ring of an algebraic variety with torus action},
        date={2010},
     journal={Advanced Math.},
      volume={225},
       pages={977\ndash 1012},
}

\bib{hw_cpl1}{article}{
      author={Hausen, Jürgen},
      author={Wrobel, Milena},
       title={Non-complete rational t-varieties of complexity one},
        date={2017},
     journal={Mathematische Nachrichten},
      volume={290},
      number={5-6},
       pages={815\ndash 826},
}

\bib{h_quadrics}{article}{
      author={Hische, Christoff},
       title={On canonical fano intrinsic quadrics},
        date={2023},
     journal={Glasgow Mathematical Journal},
      volume={65},
      number={2},
       pages={288–309},
}

\bib{hw19}{misc}{
      author={Hische, Christoff},
      author={Wrobel, Milena},
       title={On the anticanonical complex},
        date={2019},
         url={https://arxiv.org/abs/1808.01997},
}

\bib{kasNil}{article}{
      author={Kasprzyk, Alexander},
      author={Nill, Benjamin},
       title={Weighted projective spaces and reflexive simplices},
        date={2012},
     journal={Electronic Journal of Combinatorics},
      volume={19},
}

\bib{MaclaganSturmfels}{book}{
      author={Maclagan, Diane},
      author={Sturmfels, Bernd},
       title={Introduction to {T}ropical {G}eometry},
      series={Graduate Studies in Mathematics},
   publisher={American Mathematical Society, Providence, RI},
        date={2015},
      volume={161},
}

\bib{tev}{article}{
      author={Tevelev, Jenia},
       title={Compactifications of subvarieties of tori},
        date={2007},
        ISSN={0002-9327},
     journal={Amer. J. Math.},
      volume={129},
      number={4},
       pages={1087\ndash 1104},
}

\end{biblist}
\end{bibdiv}

\end{document}